\documentclass[10pt,twoside,reqno,a4paper]{amsart}
\usepackage{mathptmx}
\usepackage{amsmath}     
\usepackage{amssymb}
\usepackage{array}
\usepackage{geometry}
\usepackage[bookmarks=true,colorlinks=true, pdfstartview=FitV, linkcolor=black, citecolor=blue, urlcolor=black]{hyperref}
\usepackage{movie15}

\usepackage{color}
\definecolor{DarkRed}{rgb}{0.55,.00,0.2}
\definecolor{DarkBlue}{rgb}{0.25,.00,0.75}
\definecolor{DarkGrey}{rgb}{0.35,.35,0.35}

\newtheorem{theorem}{Theorem}[section]
\newtheorem{lemma}[theorem]{Lemma}
\newtheorem{proposition}[theorem]{Proposition}
\newtheorem{corollary}[theorem]{Corollary}

\theoremstyle{definition}

\theoremstyle{remark}
\newtheorem{remark}[theorem]{Remark}

\numberwithin{equation}{section}

\newcommand{\ds}{\displaystyle}
\newcommand{\e}{{\rm e}}

\newcommand{\seqnum}[1]{\href{http://www.research.att.com/cgi-bin/access.cgi/as/~njas/sequences/eisA.cgi?Anum=#1}{\underline{#1}}}

\begin{document}

\title[Central factorials under the KL-transform of polynomials]{Central factorials under the Kontorovich-Lebedev transform of polynomials
}

\author{Ana F. Loureiro}
\address{Centro de Matem\'atica da Universidade do Porto, Fac. Sciences of University of Porto,Rua do Campo Alegre,  687; 4169-007 Porto (Portugal)}
\thanks{Work of AFL supported by Funda\c c\~ao para a Ci\^encia e Tecnologia via the grant SFRH/BPD/63114/2009. \\ 
Research partially funded by the European Regional Development Fund through the programme COMPETE and by the Portuguese Government through the FCT (Funda\c c\~ao para a Ci\^encia e a Tecnologia) under the project PEst-C/MAT/UI0144/2011.}

\author{S. Yakubovich}
\address{Department of Mathematics, Fac. Sciences of University of Porto, Rua do Campo Alegre,  687; 4169-007 Porto (Portugal)}
\email{ \{anafsl, syakubov\}@fc.up.pt}



\date{February 2, 2012}


\begin{abstract} We show that slight modifications of the Kontorovich-Lebedev transform lead to an automorphism of the vector space of polynomials. 
This circumstance along with the Mellin transformation property of the modified Bessel functions perform the passage of monomials  to central factorial polynomials. A special attention is driven to the polynomial sequences whose KL-transform is the canonical sequence, which will be fully characterized. Finally, new identities between the central factorials and the Euler polynomials are found. 
\end{abstract}

\maketitle

\section{Introduction and preliminary results}

Throughout  the text, $\mathbb{N}$ will denote the set of all positive integers, $\mathbb{N}_{0}=\mathbb{N}\cup \{0\}$, whereas $\mathbb{R}$ and $\mathbb{C}$  the field of the real and complex numbers, respectively. The notation $\mathbb{R}_{+}$   corresponds to the set of all positive real numbers. The present investigation is primarily targeted at analysis of sequences of polynomials whose degrees equal its order, which will be shortly called as PS. Whenever the leading coefficient of each of its polynomials equals $1$, the PS is said to be a MPS ({\it monic polynomial sequence}). A PS or a MPS forms a basis of the vector space of polynomials with coefficients in $\mathbb{C}$, here denoted as $\mathcal{P}$. The convention $\prod_{\sigma=0}^{-1}:=1$ is assumed. Further notations are introduced as needed. 

We will show that, upon slight modifications on the {\it Kontorovich-Lebedev transform} (hereafter, we will shortly call {\it KL-transform}), introduced in \cite{KLarticle},  permit to transform the canonical polynomial sequence $\{x^n\}_{n\geqslant 0}$ into the so called central factorials of even or odd order \cite{Riordan2}: 
$$
	\left(x-\frac{n}{2}+\frac{1}{2}\right)_{n}
	= \left\{\begin{array}{lcl}
		(-1)^k (1-x)_{k} (1+x)_{k} &\text{ if }& n=2k  \vspace{+0.2cm}\\
		\ds (-1)^k \left(\tfrac{1}{2}-x\right)_{k} \left(\tfrac{1}{2} +x\right)_{k} &\text{ if }& n=2k+1
	\end{array}\right.
$$
where the $(x)_{n}$ represents the {\it Pochammer symbol}: $(x)_{n}:=\prod\limits_{\sigma=0}^{n-1}(x+\sigma)$ when $n\geqslant 1$ and $(x)_{0}=1$. Indeed the set $\{\left(x-\frac{n}{2}+\frac{1}{2}\right)_{n}\}_{n\geqslant 0}$ is an Appell sequence with respect to the central difference operator $\delta$, defined by $(\delta f)(x)=f(x+\frac{1}{2})-f(x-\frac{1}{2})$, for any $f\in\mathcal{P}$  \cite{Charalambides, Riordan2}, since $\delta \left(x-\frac{n}{2}-\frac{1}{2}\right)_{n+1}= (n+1) \left(x-\frac{n}{2}+\frac{1}{2}\right)_{n}$.

Precisely, we define the two following modifications of the {\it Kontorovich-Lebedev transform}, which figure out to be our main tools \cite{LebedevKL,Sneddon,Yakubovich2003,YakubovichBook1996}
\begin{equation} \label{KL s directa}
	KL_{s}[f](\tau) =\frac{2 \sinh (\pi\sqrt{\tau})}{\pi\sqrt{\tau}}
		\int_{0}^{\infty} K_{2i\sqrt{\tau}}(2\sqrt{x}) f(x) dx \ , 
\end{equation}
and 
\begin{equation} \label{KL c directa}
	KL_{c}[f](\tau) =\frac{2 \cosh (\pi\sqrt{\tau})}{\pi }
		\int_{0}^{\infty} K_{2i\sqrt{\tau}}(2\sqrt{x}) f(x) \frac{dx}{\sqrt{x}}\ , 
\end{equation}
where $K_{\nu}(z)$ represents the modified Bessel function (also known as Macdonald function) \cite[Vol.II]{Bateman}\cite{Lebedev}.  
The reciprocal inversion formulas are, respectively, 
\begin{equation}\label{KL s inverse}
	x\ f(x)= \frac{2}{\pi} \lim_{\lambda\to \pi -} 
		\int_{0}^{\infty} \sqrt{\tau} \cosh(\lambda\sqrt{\tau}) K_{2i\sqrt{\tau}}(2\sqrt{x})
		\ KL_{s}[f](\tau)d\tau
\end{equation}
and 
\begin{equation}\label{KL c inverse}
	\sqrt{x}\ f(x)= \frac{2}{\pi} \lim_{\lambda\to \pi -} 
		\int_{0}^{\infty}  \sinh(\lambda\sqrt{\tau}) K_{2i\sqrt{\tau}}(2\sqrt{x})
		\ KL_{c}[f](\tau)d\tau
\end{equation}
The formulas \eqref{KL s directa}-\eqref{KL c inverse}  are valid for any continuous function  $f \in L_{1}\left(\mathbb{R}_{+}, K_{0}(2\mu\sqrt{x}) dx \right)$, $0<\mu<1$, in a neighborhood of each $x\in\mathbb{R}_{+}$ where $f(x)$ has bounded variation  \cite[Th. 6.3]{YakuLuchko}. Properties about KL-transforms in $L_{p}$-spaces can be found in \cite{Yakubovich2004JAT} and \cite{YakubovichFisher1994}. 

The kernel of such transformation is the modified Bessel function (also called {\it MacDonald function}) $K_{2i\sqrt{\tau}}(2\sqrt{x})$ of purely imaginary index, which is real valued and can be defined by integrals of Fourier type 
\begin{align}\label{Cosine Fourier K}
	& K_{2i\sqrt{\tau} }(2\sqrt{x}) = \int_{0}^\infty \e^{-2\sqrt{x} \cosh(u)}\cos(2\sqrt{\tau} \, u)du \ , \ x\in\mathbb{R}_{+}, \ \tau\in\mathbb{R}_{+} .
\end{align}
Moreover it is an eigenfunction of the operator 
\begin{equation}\label{op A }
	\mathcal{A} \ = \ x^{2} \frac{d^{2}}{dx^{2}} + x \frac{d}{dx} -x  \ = \  x \frac{d}{dx} x \frac{d}{dx} -x
\end{equation} 
insofar as 
\begin{equation}\label{A Kitau}
	\mathcal{A} K_{2 i \sqrt{\tau} }(2\sqrt{x}) = -\tau \ K_{2i\sqrt{\tau}}(2\sqrt{x}) \ .
\end{equation}

Besides, $K_{\nu}(2\sqrt{x})$ reveals the asymptotic behaviour with respect to $x$ \cite[Vol. II]{Bateman}\cite{YakubovichBook1996}
\begin{align}
	& K_{\nu}(2\sqrt{x}) = \frac{\sqrt{\pi}}{2 \ x^{1/4}}  \e^{-2\sqrt{x}}[1+O(1/\sqrt{x})] , \quad x\rightarrow +\infty,
		\label{Knu at infty}\\
	& K_{\nu}(2\sqrt{x}) =O(x^{-\Re(\nu)/2}) \ , \ K_{0}(2\sqrt{x}) =O(\log x)  \ , \quad x\rightarrow 0. \label{K0 at 0}
\end{align}

From the comparison between \eqref{KL s directa} and \eqref{KL c directa}, readily comes out the identity 
$$
	KL_{s}[f](\tau) = \frac{\tanh(\pi\sqrt{\tau})}{\sqrt{\tau}}  KL_{c}[\sqrt{x} f(x)](\tau) \ .
$$

The KL-transform of the canonical sequence $\{x^{n}\}_{n\geqslant 0}$, is also a MPS whose elements are the {\it central factorials}. Indeed, while evaluating the Mellin transform of the function $K_{i\tau}(2\sqrt{x})$  at positive integer values ({\it i.e.}, the moments of this function), the output is a product of an elementary function by a polynomial whose degree is exactly the order of the moment increased by one unity. To be more specific, for positive real values of $\tau $, we recall relation (2.16.2.2) in \cite{PrudnikovMarichev}  
\begin{align}
	& \label{KL s xn}
	\begin{array}{lcl}
	KL_{s}[x^n](\tau)& =& \ds  
	\frac { 2\sinh(\pi \sqrt{\tau})}{\pi\sqrt{\tau}} \int_{0}^\infty K_{2i\sqrt{\tau}}(2\sqrt{x}) x^n dx 
	= \prod_{\sigma=1}^n \left(\sigma^2 + \tau\right) \\
	&=& \ds \left(1- i\sqrt{\tau}\right)_{n}\left(1+ i\sqrt{\tau}\right)_{n}
	\ ,\  n\in\mathbb{N}_{0},
	\end{array} \\
	& \label{KL c xn}
	\begin{array}{lcl}
	KL_{c}[x^n](\tau) &=& \ds 
	\frac { 2\cosh(\pi \sqrt{\tau})}{\pi} \int_{0}^\infty K_{2i\sqrt{\tau}}(2\sqrt{x}) x^{n-\frac{1}{2}} dx 
	= \prod_{\sigma=0}^{n-1} \left(\left(\tfrac{1}{2}+\sigma \right)^2 + \tau \right)\\
	&=& \ds \left(\tfrac{1}{2}- i\sqrt{\tau}\right)_{n}\left(\tfrac{1}{2}+ i\sqrt{\tau}\right)_{n}
	\ ,\qquad   n\in\mathbb{N}_{0}, 
	\end{array}
\end{align}
while as $\tau \to 0$, the moments become much more simpler
\begin{equation}\label{MomentsK0}
	KL_{s}[x^n](0)=   (n!)^{2}\ , \  n\in\mathbb{N}_{0}.
\end{equation}

As a matter of fact, from \eqref{KL s xn} and of \eqref{KL c xn}, we observe that the outcome of the $KL_{s}$ and $KL_{c}$ transforms of the canonical MPS $\{x^{n}\}_{n\geqslant 0}$ are two other basis of $\mathcal{P}$ formed by the so-called {\it central factorial polynomials} (or shrotly, {\it the central factorials}) of even order $\{\left(1- i\sqrt{x}\right)_{n}\left(1+ i\sqrt{x}\right)_{n}\}_{n\geqslant 0}$ and those of odd order $\{\left(\tfrac{1}{2}- i\sqrt{x}\right)_{n}\left(\tfrac{1}{2}+ i\sqrt{x}\right)_{n}\}_{n\geqslant 0}$. Therefore, as it will be explained in \S \ref{sec: KL of a MPS}, the two transforms $KL_{s}$ and $KL_{c}$ actually behave as an isomorphic operator representing the passage between the monomials and central factorials. These latter were treated in the  book by Riordan  \cite[pp.212-217,233-236]{Riordan2} while analyzing the central difference operator $\delta$. The connection coefficients between $\{x^{n}\}_{n\geqslant 0}$ and each one of the basis $\{\left(1- i\sqrt{x}\right)_{n}\left(1+ i\sqrt{x}\right)_{n}\}_{n\geqslant 0}$ and $\{\left(\tfrac{1}{2}- i\sqrt{x}\right)_{n}\left(\tfrac{1}{2}+ i\sqrt{x}\right)_{n}\}_{n\geqslant 0}$ correspond to the {\it central factorial numbers} of even and odd order, respectively. They have been appeared in many different contexts from the approximation theory \cite{Butzer,Butzer2} to algebraic geometry \cite{application1,application2}, not disregarding spectral theory of differential operators \cite{Loureiro2010,EverittJacobiStirling}.  

Entailed in this framework, we bring to light two MPSs, $\{P_{n}\}_{n\geqslant0}$ and $\{\widetilde{P}_{n}\}_{n\geqslant0}$, having the canonical sequence as the corresponding $KL_{s}$ and $KL_{c}$-transform. The characterization of these MPS is the main goal of the present work and it will be unraveled throughout \S\ref{sec: KL of Pn}, where the  {\it central factorial numbers} are an asset. But the analytical properties of the $KL_{s,c}$-transform interlaced with algebraic properties of the polynomial sequences are indeed the fabric of almost all the developments hereby made. Thus, after obtaining generating functions for $\{P_{n}\}_{n\geqslant0}$ and $\{\widetilde{P}_{n}\}_{n\geqslant0}$ on \S\ref{subsec: gf}, we will be focusing, in \S\ref{subsec: rec rel Pn Pntilde}, on the integral and algebraic relations between the aforementioned polynomial sequences and the Euler polynomials. Precisely, such connection is indeed the key ingredient not only to recognize the connection coefficients between $\{P_{n}\}_{n\geqslant0}$ and $\{\widetilde{P}_{n}\}_{n\geqslant0}$, but also to grasp the recurrence relation fulfilled by each of these sequences. 

In \S\ref{sec: Prudnikov} we analyze the behavior of the corresponding dual sequences. Generally speaking, the dual sequence $\{v_{n}\}_{n\geqslant 0}$ of a given MPS $\{Q_{n}\}_{n\geqslant 0}$ belong to the dual space $\mathcal{P}'$ of $\mathcal{P}$ and whose elements are uniquely defined by 
$$
	\langle v_{n},Q_{k}  \rangle := \delta_{n,k}, \; n,k\geqslant 0,
$$ 
where $\delta_{n,k}$ represents the {\it Kronecker delta} function. Its first element, $u_{0}$, earns the special name of  {\it canonical form} of the MPS. Here, by $\langle u,f\rangle$ we mean the action of $u\in\mathcal{P}'$ over $f\in\mathcal{P}$, but a special notation is given to the action over the elements  of the canonical sequence $\{x^{n}\}_{n\geqslant 0}$ -- the {\it moments of $u\in\mathcal{P}'$}:   $(u)_{n}:=\langle u,x^{n}\rangle, n\geqslant 0 $. Whenever there is a form $v\in\mathcal{P}'$ such that $\langle v , Q_{n} Q_{m} \rangle = k_{n} \delta_{n,m}$ with $k_{n}\neq0$ for all $n,m\in\mathbb{N}_{0}$ \cite{MaroniTheorieAlg,MaroniVariations}. The PS $\{Q_{n}\}_{n\geqslant 0}$ is then said to be orthogonal with respect to $v$ and we can assume the system (of orthogonal polynomials) to be monic  and the original form $v$ is proportional to $v_{0}$. This unique MOPS $\{Q_{n}(x)\}_{n\geqslant 0}$ with respect to the regular form $v_{0}$ can be characterized by the popular second order recurrence relation 
\begin{align} \label{MOPS rec rel} 
&	\left\{ \begin{array}{@{}l}
		Q_{0}(x)=1 \quad ; \quad Q_{1}(x)= x-\beta_{0} \vspace{0.15cm}\\
		Q_{n+2}(x) = (x-\beta_{n+1})Q_{n+1}(x) - \gamma_{n+1} \, Q_{n}(x) \ , \quad n\in\mathbb{N}_{0},  
	\end{array} \right. 
\end{align}
where $\beta_{n}=\frac{\langle v_{0},x Q_{n}^2  \rangle}{\langle v_{0}, Q_{n}^2  \rangle}$ and $\gamma_{n+1}=\frac{\langle v_{0}, Q_{n+1}^2  \rangle}{\langle v_{0}, Q_{n}^2  \rangle}$ for all $n\in\mathbb{N}_{0}$. 

Although the attained recursive relations for $\{P_{n}\}_{n\geqslant0}$ and $\{\widetilde{P}_{n}\}_{n\geqslant0}$  reject their (regular) orthogonality ({\it i.e.}, with respect to an $L_{2}$-inner product) of the two MPSs, because they do not fulfill a second order recursive relation of the type \eqref{MOPS rec rel}, their corresponding canonical forms $u_{0}$ and $\widetilde{u}_{0}$ are positive definite, respectively associated to the weight functions $K_{0}(2\sqrt{x})$ and $\frac{1}{\sqrt{x}}K_{0}(2\sqrt{x})$. Therefore, the existence of two MOPSs $\{Q_{n}\}_{n\geqslant0}$ and $\{\widetilde{Q}_{n}\}_{n\geqslant0}$ with respect to  $u_{0}$ and $\widetilde{u}_{0}$, respectively, is ensured. The problem of characterizing the first one was posed by Prudnikov in \cite{PrudnikovProblem} is still open.

As a consequence of the developments made, new identities involving the Genocchi numbers, the central factorial and the Euler numbers and polynomials will emerge. Thus on \S\ref{Sec: Number identities}
we bring this humble contribution to those already known \cite{Abramowitz, Chang2009,Charalambides,Cigler,Cvijovic,Liu2003,MacMahon,Riordan2,StanleyVol2,WZhang}, not disregarding \cite{OEIS}. 

Finally, we trace some of intriguing questions that we believe to be worth to be explained.

\section{The KL-transform of a polynomial sequence}\label{sec: KL of a MPS}

The central role played by the central factorials throughout this work gives reasons to begin by reviewing  their foremost properties. But first we will visualize them as a whole to afterwards split the analysis into the even and odd order cases. Thus, the set of central factorial polynomials $\{\left( x- \frac{n}{2} +\frac{1}{2}\right)_{n} \}_{n\geqslant 0}$ is a MPS (ergo, form a basis of $\mathcal{P}$) bridged to the canonical MPS  $\{x^{n} \}_{n\geqslant 0}$ via  the {\it central factorial numbers of first and second kind} $\{\left(t(n,\nu),T(n,\nu)\right)\}_{0\leqslant \nu\leqslant n}$ \cite{Riordan2}: 
\begin{equation}\label{central fact polyn}
\begin{array}{ccll}
    \ds \left( z- \frac{n}{2} +\frac{1}{2}\right)_{n}   
    	&=& \ds \sum_{\nu =0}^n (-1)^{n+\nu} t(n+1,\nu+1) \: z^\nu 
    	&\:,\;n\in\mathbb{N}_{0}, \vspace{0.0cm}     \\
    z^n    &=& \ds  
    	\sum_{\nu =0}^n (-1)^{n+\nu} T(n+1,\nu+1) \: \left( z- \frac{\nu}{2} +\frac{1}{2}\right)_{\nu} &\ ,\;
    n\in\mathbb{N}_{0},
\end{array}
\end{equation}
The central factorial numbers of first kind $t(n,\nu)$ fulfill the triangular relation 
\begin{equation}\label{central fact1 RecRelation}
\left\{\begin{array}{rcl}
	 t(n,\nu)
        = t(n-2,\nu-2) - \frac{1}{4}(n-2)^{2}\  t(n-2,\nu) 
        &,& 0\leqslant \nu\leqslant n, 
           \vspace{0.2cm}\\
	 t (n,0)= t (0,n)=\delta_{n,0} &,& n\geqslant 0,
\end{array}\right. 
\end{equation}
whereas those of second kind $T(n,\nu)$ satisfy 
\begin{equation}\label{central fact2 RecRelation}
\left\{\begin{array}{rcl}
	 T(n,\nu)
        = T(n-2,\nu-2) - \frac{1}{4}\nu^{2}\  T(n-2,\nu) 
        &,& 0\leqslant \nu\leqslant n, 
           \vspace{0.2cm}\\
	 T (0,n)=T(0,n)=\delta_{n,0}\ , \ T(n,n)=1 &,& n\geqslant 0,
\end{array}\right. 
\end{equation}
but whenever $\ \nu\geqslant n+1 \ \text{or} \ (-1)^n+(-1)^\nu=0 $, necessarily, $t(n,\nu)=T (n,\nu)=0$, impelling to  split the analysis into the cases of  even or odd order.  Following up the idea, we will adopt the notation 
$$
	\left\{\begin{array}{c}
	t_{E}(n,\nu):=t(2n,2\nu)  \\
	T_{E}(n,\nu):= T(2n,2\nu) 
	\end{array} \right.  
	\qquad \text{ and }\qquad 
	\left\{\begin{array}{c}
	t_{O}(n,\nu) := t(2n+1,2\nu+1)  \\
	T_{O}(n,\nu) := T(2n+1,2\nu+1)
	\end{array}\right. 
	\ , \quad 0\leqslant \nu\leqslant n. 
$$
Thus, according to the aforementioned properties, the relations \eqref{central fact polyn} straightforwardly supply  
\begin{eqnarray}\label{central fact change basis1}
&& \label{x k stirl1 x k,A}
    \left\{\begin{array}{c}
	(1-i\sqrt{\tau})_{n}(1+i\sqrt{\tau})_{n} 		\vspace{0.2cm}
	\\
	\left(\tfrac{1}{2}- i\sqrt{\tau}\right)_{n}\left(\tfrac{1}{2}+ i\sqrt{\tau}\right)_{n}
	\end{array}\right\}
    	= \sum_{\nu =0}^n (-1)^{n+\nu} 
	\left\{\begin{array}{c}
	t_{E}(n+1,\nu+1)		\vspace{0.2cm} \\
	t_{O}(n,\nu)
	\end{array}\right\}  \: \tau^\nu \:,\;n\in\mathbb{N}_{0},
\end{eqnarray}
and, reciprocally, 
\begin{eqnarray}\label{central fact change basis2}
&& \label{x k Stirl2 x k,A}
    \tau^n     = \sum_{\nu =0}^n (-1)^{n+\nu} 
    				\left\{\begin{array}{c}
				T_{E}(n+1,\nu+1) \:(1-i\sqrt{\tau})_{\nu}
    				(1+i\sqrt{\tau})_{\nu} 		\vspace{0.2cm} \\
				T_{O}(n,\nu) \:
				\left(\tfrac{1}{2}- i\sqrt{\tau}\right)_{\nu}\left(\tfrac{1}{2}+ i\sqrt{\tau}\right)_{\nu}
				\end{array}\right\} 
    				\quad ,\quad n\in\mathbb{N}_{0}.
\end{eqnarray}

The closed form expression for the central factorial numbers of second kind (even or odd) can be read on \cite[p.214]{Riordan2}
\begin{equation*}
	T(n,\nu)= \frac{1}{\nu!} \sum_{\mu=0}^{\nu} \binom{\nu}{\mu} (-1)^{\mu}
            \left(\frac{1}{2}\nu - \mu \right)^{n}
            \ , \quad 0\leqslant \nu\leqslant n. 
\end{equation*}
For further readings we refer to the entries \seqnum{A008955} and \seqnum{A008956} at \cite{OEIS}. 

Apparently the even order case appear more frequently than the odd order one. For instance, in \cite{Loureiro2010,EverittJacobiStirling} while dealing with spectral properties of even order differential operators having the classical orthogonal polynomials as eigenfunctions, the central factorial numbers appear as particular cases of the so-called $z$-{\it modified Stirling numbers} \cite{Loureiro2010}  or Jacobi-Stirling numbers \cite{EverittJacobiStirling} Notwithstanding they are just a particular case of the $z$-modified Stirling numbers, the central factorial numbers of even order interfere in the expression of the expansion of the $z$-modified Stirling numbers in terms of powers of $z$ - see \cite{Zeng}. They have indeed received combinatorial interpretations from different aspects \cite{Andrews, Zeng,Gessel,Mongeli}.

Among the known relations between the central factorial numbers and other well known sets of numbers we recall  \cite[p.824]{Abramowitz}\cite[(30), p. 216]{Riordan2}, namely 
$$
	T(n,\nu) = \sum_{\mu=0}^{n-\nu} \binom{n}{\mu} S(n-\mu,\nu) \left(-\frac{1}{2} \nu\right)^{\mu}
	\ , 0\leqslant \nu \leqslant n\ , \ n,\nu\in\mathbb{N}_{0}.
$$
For connections with the higher order tangent and secant numbers we refer to \cite{Chang2009,Cvijovic}.

\begin{proposition} \label{prop: Isomorphism}The $KL_{s}$ and $KL_{c}$-transforms are both an automorphism of the vector space of polynomials defined on $\mathbb{R}_{+}$. 
\end{proposition}

\begin{proof} Since $\{x^n\}_{n\geqslant 0}$,  $\{\left(1- i\sqrt{x}\right)_{n}\left(1+ i\sqrt{x}\right)_{n}\}_{n\geqslant 0}$ and $\{\left(\tfrac{1}{2}- i\sqrt{x}\right)_{n}\left(\tfrac{1}{2}+ i\sqrt{x}\right)_{n}\}_{n\geqslant 0}$ form a basis of the vector space of polynomials defined on $\mathbb{R}_{+}$, then, on account of \eqref{KL s xn}-\eqref{KL c xn}, the linearity together with the injectivity of the $KL_{s}$ and $KL_{c}$ transforms ensures the result. 
\end{proof}

\begin{remark} Both of the modified KL-transforms under analysis can be extended to any continuous function $f \in L_{1}\left(\mathbb{R}, K_{0}(2\mu\sqrt{|x|}) dx \right)$, $0<\mu<1$, in a neighborhood of each $x\in\mathbb{R}$ where $f(x)$ has bounded variation in the following manner \cite{YakuLuchko}
$$
	KL_{s}[f(x)](\tau) = \frac{\sinh(\pi\sqrt{|\tau|})}{\pi\sqrt{|\tau|}}
		\int_{-\infty}^{+\infty} K_{2i\sqrt{|\tau|}}(2\sqrt{|x|}) f(|x|) dx
$$
and 
$$
	KL_{c}[f(x)](\tau) = \frac{\cosh(\pi\sqrt{|\tau|})}{\pi}
		\int_{-\infty}^{+\infty} K_{2i\sqrt{|\tau|}}(2\sqrt{|x|}) f(|x|) \frac{dx}{\sqrt{|x|}} \ .
$$
Consequently, the latter result can be extended to the vector space of polynomials $\mathcal{P}$ without any further restrictions over the domain.
\end{remark}

\begin{lemma} \label{lem: KL sc Am xk Bn} For any MPS  $\{B_{n}\}_{n\geqslant0}$ and any two integers $k,m\in\mathbb{N}_{0}$, we have 
	\begin{align} 
	& \label{KL s Ak Pn}
	KL_{s}\left[\frac{1}{x}\mathcal{A}^m x^{k+1} B_{n}(x)\right](\tau)
		=  (-1)^m \tau^{m} KL_{s}[x^{k} B_{n}](\tau)
	\ , \ n\in\mathbb{N}_{0}, \\
	& \label{KL c Ak Pn tilde}
	KL_{c}\left[\frac{1}{\sqrt{x}}\mathcal{A}^m x^{k+\frac{1}{2}} B_{n}(x)\right](\tau)
		=  (-1)^m \tau^{m} KL_{c}[x^{k} B_{n}](\tau)
	\ , \ n\in\mathbb{N}_{0},
	\end{align}
and 
	\begin{align} \label{Ak Pn invKL}
	& \mathcal{A}^m  x^{k+1} B_{n} (x)=(-1)^m \frac{2}{\pi} \lim_{\lambda\to {\pi}-} 
	 			\int_{0}^\infty \cosh(\lambda \sqrt{\tau}) K_{2i\sqrt{\tau}}(2\sqrt{x}) 
				  \tau^{m+\frac{1}{2}} KL_{s}[x^{k} B_{n}](\tau) d\tau
	\ , \ n\in\mathbb{N}_{0}, \\
	&  \label{Ak Bn invKL c}
	\mathcal{A}^m x^{k+\frac{1}{2}} B_{n} (x)
		=(-1)^m \frac{2}{\pi} \lim_{\lambda\to {\pi}-} 
	 			\int_{0}^\infty \sinh(\lambda \sqrt{\tau}) K_{2i\sqrt{\tau}}(2\sqrt{x}) 
				  \tau^{m} KL_{c}[x^{k} B_{n}](\tau) d\tau
	\ , \ n\in\mathbb{N}_{0},
	\end{align}
where $\mathcal{A}$ represents the operator in \eqref{op A } acting over the variable $x$. 
\end{lemma}

\begin{proof} Following a similar procedure of the one taken in  \cite{Yakubov2002}, we come out with
\begin{equation}\label{int by parts}
	\int_{0}^\infty \psi(x) \Big(\mathcal{A}\varphi(x)\Big) \frac{dx}{x}
	=	\int_{0}^\infty  \Big(\mathcal{A} \psi(x)\Big) \varphi(x) \frac{dx}{x},
\end{equation}
whenever $\phi,\psi\in\mathcal{C}_{0}^{2}(\mathbb{R}_{+})$ vanishing at $\infty$ and near the origin together with their derivatives, in order to eliminate the outer terms. 
Thus, due to \eqref{A Kitau}, we can successively write  
$$\begin{array}{lcl}
	\ds \int_{0}^\infty K_{2i\sqrt{\tau}}(2\sqrt{x})  \Big(\mathcal{A}^{m} x^{k+1} B_{n}(x)\Big) \frac{dx}{x}
	&=& \ds (-1)^m \tau^{m} \int_{0}^\infty K_{2i\sqrt{\tau}}(2\sqrt{x}) x^{k+1} B_{n}(x)  \frac{dx}{x}\\
	&=&\ds (-1)^m \tau^{m} \int_{0}^\infty K_{2i\sqrt{\tau}}(2\sqrt{x}) x^{k} B_{n}(x)  {dx}
	\ ,\   n\in\mathbb{N}_{0},
	\end{array}
$$
which provides \eqref{KL s Ak Pn}. Likewise, considering that 
$$\begin{array}{lcl}
	\ds \int_{0}^\infty K_{2i\sqrt{\tau}}(2\sqrt{x})  
		\frac{1}{\sqrt{x}}\Big(\mathcal{A}^{m} x^{k+\frac{1}{2}} B_{n}(x)\Big) \frac{dx}{\sqrt{x}}
	&=& \ds (-1)^m \tau^{m} \int_{0}^\infty K_{2i\sqrt{\tau}}(2\sqrt{x}) x^{k+\frac{1}{2}} B_{n}(x)  \frac{dx}{x}\\
	&=&\ds (-1)^m \tau^{m} \int_{0}^\infty K_{2i\sqrt{\tau}}(2\sqrt{x}) x^{k} B_{n}(x)  \frac{dx}{\sqrt{x}}
	\ ,\   n\in\mathbb{N}_{0},
	\end{array}
$$
we deduce \eqref{KL c Ak Pn tilde} holds. Finally, on the grounds of \eqref{KL s directa}-\eqref{KL s inverse}, the relation \eqref{Ak Pn invKL} is a mere consequence of \eqref{KL s Ak Pn}, just like \eqref{Ak Bn invKL c} is a consequence of \eqref{KL c Ak Pn tilde} within the framework of \eqref{KL c directa}-\eqref{KL c inverse}.
\end{proof}

\section{On the MPS whose KL-transform is the canonical MPS}\label{sec: KL of Pn}

The relations \eqref{KL s xn}-\eqref{MomentsK0} have the straightforward, yet important, consequence that the KL transform of any MPS is again another MPS. This legitimates the question of seeking two  MPSs $\{P_{n}\}_{n\geqslant 0}$ and $\{\widetilde{P}_{n}\}_{n\geqslant 0}$ whose $KL_{s}$ and $KL_{c}$ transforms correspond, respectively, to the canonical sequence $\{\tau^{n}\}_{n\geqslant 0}$. 

\begin{proposition} The two polynomial sequences $\{P_{n}\}_{n\geqslant 0}$ and $\{\widetilde{P}_{n}\}_{n\geqslant 0}$ such that  
\begin{equation}\label{KL Pn Kl Pn tilde simple}
	KL_{s}[P_{n}(x)](\tau)
	=  KL_{c}[\widetilde{P}_{n}(x)](\tau)
	= \tau^{n} \quad , \quad n\in\mathbb{N}_{0}. 
\end{equation}
are respectively given by 
\begin{equation} \label{Pn An of x}
	{P}_{n} (x) = (-1)^{n} \frac{1}{x} \mathcal{A}^{n} x \quad , \quad n\in\mathbb{N}_{0}, 
\end{equation}
\begin{equation} \label{Pntilde An of x}
	\widetilde{P}_{n} (x) = (-1)^{n} \frac{1}{\sqrt{x}} \mathcal{A}^{n} \sqrt{x} \quad , \quad n\in\mathbb{N}_{0}. 
\end{equation}
Moreover, $\{P_{n}\}_{n\geqslant 0}$ and $\{\widetilde{P}_{n}\}_{n\geqslant 0}$ necessarily fulfill    
\begin{equation}\label{PolySeq Pn}
	P_{n+1}(x)= - x^{2} P_{n}''(x) - 3x P_{n}'(x) - (1-x) P_{n}(x) \ , \quad  n\in\mathbb{N}_{0},
\end{equation}
\begin{equation}\label{PolySeq Pntilde}
	\widetilde{P}_{n+1}(x)= - x^{2} \widetilde{P}_{n}''(x) - 2x \widetilde{P}_{n}'(x) 
				- \left(\frac{1}{4}-x\right) \widetilde{P}_{n}(x) \ , \quad  n\in\mathbb{N}_{0},
\end{equation}
with $P_{0}(x)=1=\widetilde{P}_{0}(x)$, 
whose elements are explicitly given by 
\begin{align}\label{Pn explicit exp}
	& P_{n}(x) = \sum_{\nu=0}^{n} (-1)^{n+\nu} T_{E}(n+1,\nu+1) x^{\nu}
	\quad , \quad n\in\mathbb{N}_{0}, \\
	& \label{Pntilde explicit exp}
	\widetilde{P}_{n}(x) = \sum_{\nu=0}^{n} (-1)^{n+\nu} T_{O}(n,\nu) x^{\nu}
	\quad , \quad n\in\mathbb{N}_{0}. 
\end{align}
\end{proposition}

\begin{proof}  In the light of Proposition \ref{prop: Isomorphism}, the existence and uniqueness of two MPSs  $\{P_{n}\}_{n\geqslant 0}$ and $\{\widetilde{P}_{n}\}_{n\geqslant 0}$ such that  \eqref{KL Pn Kl Pn tilde simple} holds is guaranteed. Now, upon the suitable replacement of $k=n=0$, Lemma \ref{lem: KL sc Am xk Bn} ensures that 
$$
	KL_{s}\left[\frac{(-1)^{n}}{x}\mathcal{A}^{n} x \right](\tau) = \tau^{n}
	\qquad \text{and} \qquad 
	KL_{c}\left[ \frac{(-1)^{n}}{\sqrt{x}}\mathcal{A}^{n} \sqrt{x} \right](\tau)
	= \tau^{n} \ , \quad n\in\mathbb{N}_{0}.
$$
The fact that $KL_{s}$ and $KL_{c}$ are two isomorphic transformations in $\mathcal{P}$ qualify both sequences $\{\frac{(-1)^{n}}{x}\mathcal{A}^{n} x\}_{n\geqslant 0}$  and $\{ \frac{(-1)^{n}}{\sqrt{x}}\mathcal{A}^{n} \sqrt{x}\}_{n\geqslant 0}$ as two MPSs. 

Indeed, if $P_{n}(x)=\frac{(-1)^{n}}{x}\mathcal{A}^{n} x$ and $\widetilde{P}_{n}(x)=\frac{(-1)^{n}}{\sqrt{x}}\mathcal{A}^{n} \sqrt{x}$, then 
$$
		P_{n+1}(x)	=   (-1)^{n+1} \frac{1}{x} \mathcal{A} \ x \  \frac{1}{x}  \ \mathcal{A}^{n} x 
					= - \frac{1}{x} \mathcal{A} \Big( x P_{n}(x)\Big) 
					= - x^{2} P_{n}''(x) - 3x P_{n}'(x) - (1-x) P_{n}(x) \ , \quad n\in\mathbb{N}_{0}, 
$$ 
with $P_{0}(x)=1$, while 
$$
	\widetilde{P}_{n+1}(x)  =   \frac{(-1)^{n+1}}{\sqrt{x}}\mathcal{A}^{n+1} \sqrt{x}
			=    \frac{-1}{\sqrt{x}}\mathcal{A}\left(  \sqrt{x} \widetilde{P}_{n}(x)\right)
			= - \left\{ x^{2} \widetilde{P}_{n}''(x) + 2 x \widetilde{P}_{n}'(x)  
			+ (\tfrac{1}{4} -x) \widetilde{P}_{n}(x) \right\} \ , \ n\in\mathbb{N}_{0}, 
$$
with $\widetilde{P}_{0}(x)=1$. 
By equating the first and last members of the two latter equalities, we respectively obtain \eqref{PolySeq Pn}  and \eqref{PolySeq Pntilde}. The first one is fulfilled by the MPS $\{{P}_{n}\}_{n\geqslant 0}$ represented $P_{n}(x) = \sum\limits_{k=0}^{n} c_{n,k} \, x^{k}$ as long as the coefficients $c_{n,k}$ satisfy the relation 
$$
	c_{n+1,k} = c_{n,k-1} -(k+1)^2 c_{n,k} \quad , \quad 0\leqslant k\leqslant n, \quad n,k\in\mathbb{N}_{0},
$$
while \eqref{PolySeq Pntilde} is realized by the MPS $\{ \widetilde{P}_{n}(x)\}_{n\geqslant 0}$ represented by $\widetilde{P}_{n}(x) = \sum\limits_{k=0}^{n} \widetilde{c}_{n,k} \, x^{k}$ if 
$$
	\widetilde{c}_{n+1,k} = \widetilde{c}_{n,k-1} -\left(k+\frac{1}{2}\right)^2 \widetilde{c}_{n,k} \quad , \quad 0\leqslant k\leqslant n, \quad n,k\in\mathbb{N}_{0},
$$
under the convention $c_{n,-1}=c_{n,k}=0=\widetilde{c}_{n,-1}=\widetilde{c}_{n,k}$ whenever $k>n$ and with the initial conditions $c_{n,0}=(-1)^{n}$ and $\widetilde{c}_{n,0}=(-\frac{1}{4})^{n}$ for $n\in\mathbb{N}$ and $c_{n,n}=1=\widetilde{c}_{n,n}$ for $n\in\mathbb{N}_{0}$. Recalling \eqref{central fact2 RecRelation}, necessarily $c_{n,k}=(-1)^{n+k} T_{E}(n+1,k+1)$ whereas $\widetilde{c}_{n,k}=(-1)^{n+k} T_{O}(n+1,k+1)$ for any $n,k\in\mathbb{N}_{0}$, whence \eqref{Pn explicit exp}-\eqref{Pntilde explicit exp}.

Conversely, the two relations \eqref{Pn explicit exp} and \eqref{Pntilde explicit exp} imply \eqref{Pn An of x} and \eqref{Pntilde An of x} for $n=0,1,2$, respectively. But the remaining values for $n\in\mathbb{N}$ are also accomplished, since \eqref{central fact2 RecRelation} for even and odd values of $n$ permits to successively write 
$$
	\begin{array}{l}
		P_{n+1}(x) = \ds\sum_{k=0}^{n+1} (-1)^{n+k+1} \, 
		T_{E}(n+2,k+1) \ x^{k} = \ds  \sum_{k=0}^{n+1} (-1)^{n+k+1} \, 
		\Big( T_{E}(n+1,k) +(k+1)^{2} T_{E}(n+1,k+1)   \Big)\ x^{k} \\
		\quad =  \ds  \sum_{k=0}^{n} (-1)^{n+k}  T_{E}(n+1,k+1)    x^{k+1}
			- \sum_{k=0}^{n} (-1)^{n+k}  T_{E}(n+1,k+1)\left( \frac{d}{dx} x\frac{d}{dx} x\right) x^{k}\\
		\quad =  \ds \left( x-  \frac{d}{dx} x\frac{d}{dx} x \right) 
			P_{n}(x)
			= \frac{-1}{x}\mathcal{A} x P_{n}(x)
		, \quad n\in\mathbb{N}_{0}. 
	\end{array}
$$
and 
$$
	\begin{array}{l}
		\widetilde{P}_{n+1}(x) = \ds\sum_{k=0}^{n+1} (-1)^{n+k+1} \, 
		T_{O}(n+1,k) \ x^{k} = \ds  \sum_{k=0}^{n+1} (-1)^{n+k+1} \, 
		\Big( T_{O}(n,k-1) + \left(k+\frac{1}{2}\right)^{2} T_{O}(n,k)   \Big)\ x^{k} \\
		\quad =  \ds  \sum_{k=0}^{n} (-1)^{n+k}  T_{O}(n,k)    x^{k+1}
			- \sum_{k=0}^{n} (-1)^{n+k}  T_{O}(n,k)\left(\sqrt{x} \frac{d}{dx} x\frac{d}{dx} \sqrt{x}\right) x^{k}\\
		\quad =  \ds \left( x- \sqrt{x} \frac{d}{dx} x\frac{d}{dx} \sqrt{x} \right) 
			\widetilde{P}_{n}(x)
			=  \frac{-1}{\sqrt{x}}\mathcal{A} \sqrt{x} \widetilde{P}_{n}(x)
					, \quad n\in\mathbb{N}_{0}. 
	\end{array}
$$
By a finite induction process, it is straightforward to prove that this latter implies \eqref{Pn An of x}. Finally, by taking into account the linearity of the KL transform, \eqref{KL s xn}-\eqref{KL c xn} along with \eqref{central fact change basis2}, we conclude that \eqref{KL Pn Kl Pn tilde simple} is just a consequence of \eqref{Pn explicit exp}-\eqref{Pntilde explicit exp}. 
\end{proof}

\begin{table}[ht]
\caption{List of the first elements of  $\{P_{n}\}_{n\in\mathbb{N}_{0}}$ and  $\{\widehat{P}_{n}(x):= 4^n\widetilde{P}_{n}(x/4)\}_{n\in\mathbb{N}_{0}}$.}

{\scriptsize 
$\begin{array}{@{}l@{}}
\begin{array}{@{}l@{\ }l@{\ }l}
	P_{0}(x) &=& 1 \\
	P_{1}(x) &=& x-1 \\
	P_{2}(x) &=& x^2-5 x+1 \\
	P_{3}(x) &=& x^3-14 x^2+21 x-1 \\
	P_{4}(x) &=& x^4-30 x^3+147 x^2-85 x+1 \\
	P_{5}(x) &=& x^5-55 x^4+627 x^3-1408 x^2+341 x-1 \\
	P_{6}(x) &=& x^6-91 x^5+2002 x^4-11440 x^3+13013 x^2-1365 x+1 \\
	P_{7}(x) &=& x^7-140 x^6+5278 x^5-61490 x^4+196053 x^3-118482 x^2+5461 x-1\\
	P_{8}(x) &=&  x^8-204 x^7+12138 x^6-251498 x^5+1733303 x^4-3255330 x^3+1071799 x^2-21845 x+1 \\
	P_{9}(x) &=& x^9-285 x^8+25194 x^7-846260 x^6+10787231 x^5-46587905 x^4+53157079 x^3
		-9668036 x^2+87381 x-1 \\
\end{array}\vspace{0.3cm}
\\
\begin{array}{@{}l@{\ }l@{\ }l}
 \widehat{P}_{0}(x) &=& 1 \\
  \widehat{P}_{1}(x) &=&x-1 \\
  \widehat{P}_{2}(x) &=&x^2-10 x+1 \\
  \widehat{P}_{3}(x) &=&x^3-35 x^2+91 x-1 \\
 \widehat{P}_{4}(x) &=& x^4-84 x^3+966 x^2-820 x+1 \\
 \widehat{P}_{5}(x) &=& x^5-165 x^4+5082 x^3-24970 x^2+7381 x-1 \\
 \widehat{P}_{6}(x) &=& x^6-286 x^5+18447 x^4-273988 x^3+631631 x^2-66430 x+1 \\
  \widehat{P}_{7}(x) &=&x^7-455 x^6+53053 x^5-1768195 x^4+14057043 x^3-15857205 x^2+597871 x-1
   \\
  \widehat{P}_{8}(x)&=& x^8-680 x^7+129948 x^6-8187608 x^5+157280838 x^4-704652312 x^3+397027996
   x^2-5380840 x+1 \\
 \widehat{P}_{9}(x)&=& x^9-969 x^8+282948 x^7-30148820 x^6+1147981406 x^5-13444400190
   x^4+34924991284 x^3-9931080740 x^2+48427561 x-1 \\
\end{array}
\end{array}$

}\label{TablePn}
\end{table}

Meanwhile, the inverse relations of \eqref{Pn explicit exp} and \eqref{Pntilde explicit exp}, that is, the expression of $\{x^{n}\}_{n\geqslant 0}$ by means of $\{{P}_{n}\}_{n\geqslant 0}$ or $\{\widetilde{P}_{n}\}_{n\geqslant 0}$, can be achieved directly from the properties of the central factorial numbers $\big(t_{E}(n,\nu),T_{E}(n,\nu)\big)$ or $\big(t_{O}(n,\nu),T_{O}(n,\nu)\big)$, respectively. Thus, it follows: 
\begin{align}
	& \label{xn to Pn}
	x^n = \sum_{k=0}^n (-1)^{n+k} t_{E}(n+1,k+1) P_{k}(x) \ , \quad n\in\mathbb{N}_{0}, \\
	& \label{xn to Pntilde}
	x^n = \sum_{k=0}^n (-1)^{n+k} t_{O}(n,k) \widetilde{P}_{k}(x) \ , \quad n\in\mathbb{N}_{0}. 
\end{align}

Recently \cite{LouMarYak2011} the present authors studied the polynomial sequence $\{(-1)^{n}{\rm e}^{x}\ x^{-\alpha}\widehat{ \mathcal{A}}^{n}\  {\rm e}^{-x} x^{\alpha}\}_{n\geqslant 0}$ where $\widehat{ \mathcal{A}}-x^{2}=-\mathcal{A}-x$. The connection coefficients between this latter and the canonical sequence were essentially the set of decentralized central factorials of parameter $\alpha$, where the choice of $\alpha=0$ would give the central factorial numbers of even order. However, regarding the hard calculus involved specially from the analytical point of view, impelled a characterization farther less extensive as the one here taken for $\{{P}_{n}\}_{n\geqslant 0}$ or $\{\widetilde{P}_{n}\}_{n\geqslant 0}$. This work was performed in the sequel of \cite{Yakubovich2009}.

The characterization of all the polynomial sequences generated by integral composite powers of first order differential operators with polynomial coefficients acting on suitable analytical functions was taken in \cite{LouMar2011}.

\subsection{The generating function} \label{subsec: gf}

The elements of $\{{P}_{n}\}_{n\geqslant 0}$ or $\{\widetilde{P}_{n}\}_{n\geqslant 0}$ admit integral representations triggered by the inverse of the corresponding KL transform. While the $KL_{s}$-transform of $\{{P}_{n}\}_{n\geqslant 0}$ 
\begin{equation}\label{KL s Pn}
	KL_{s}[P_{n}(x)](\tau)
	= \frac{2}{\pi\sqrt{\tau}}\sinh(\pi\sqrt{\tau})\int_{0}^{\infty} K_{2i\sqrt{\tau}}(2\sqrt{x}) P_{n}(x)dx  
	= \tau^{n} \quad , \quad n\in\mathbb{N}_{0}. 
\end{equation}
has the inverse 
\begin{equation}\label{xPn KL inv}
	x P_{n}(x) = \frac{2}{\pi} \lim_{\lambda\to {\pi}-} 
			\int_{0}^{\infty} \tau^{n+\frac{1}{2}} 
			\cosh(\lambda \sqrt{\tau}) K_{2i\sqrt{\tau}}(2\sqrt{x}) d\tau
			\ , \ n\in\mathbb{N}_{0}, 
\end{equation}
the $KL_{c}$-transform of $\{\widetilde{P}_{n}\}_{n\geqslant 0}$
\begin{equation}\label{KL Pntilde}
	KL_{c}[\widetilde{P}_{n}(x)](\tau)
	= \frac{2\cosh(\pi\sqrt{\tau})}{\pi}\int_{0}^{\infty} K_{2i\sqrt{\tau}}(2\sqrt{x}) \widetilde{P}_{n}(x)\frac{dx}{\sqrt{x}}
	= \tau^{n} \quad , \quad n\in\mathbb{N}_{0},
\end{equation}
provides 
\begin{equation}\label{sqrtx Pn tilde KL inv}
	\sqrt{x}\ \widetilde{P}_{n}(x) = \frac{2}{\pi} \lim_{\lambda\to \pi-} 
			\int_{0}^{\infty} \tau^{n} 
			\sinh(\lambda \sqrt{\tau}) K_{2i\sqrt{\tau}}(2\sqrt{x}) d\tau
			\ , \ n\in\mathbb{N}_{0}.
\end{equation}

These integral representations \eqref{xPn KL inv} and \eqref{sqrtx Pn tilde KL inv} are actually a key ingredient either to show that  $\{P_{n}\}_{n\geqslant0}$ and $\{\widetilde{P}_{n}\}_{n\geqslant0}$ are the coefficients of the Taylor expansion of a certain elementary function, or to obtain their relation to the widely known Euler polynomials. This latter is of major importance for the achievement of the structure-recursive relations of $\{P_{n}\}_{n\geqslant0}$ and $\{\widetilde{P}_{n}\}_{n\geqslant0}$, as well as for the determination of their interim connection coefficients. 

\begin{lemma}\label{lem: Pn Pntilde derivatives of exp} The two polynomial sequence $\{P_{n}\}_{n\geqslant0}$ and $\{\widetilde{P}_{n}\}_{n\geqslant0}$ defined in \eqref{Pn An of x}-\eqref{Pntilde An of x} can be also represented by 
\begin{align} \label{repre Pn in pi over 2}
	 x P_{n}(x) 
		&= \lim_{\lambda\to{\pi}-} \frac{\partial^{2n+2}}{\partial \lambda^{2n+2}} 
		\e^{-2\sqrt{x} \cos(\lambda/2)}
		\ , \quad n\in\mathbb{N}_{0}\ , \\
	 \label{repre Pntilde in pi over 2}
	\sqrt{x} \ \widetilde{P}_{n}(x) 
		&= \lim_{\lambda\to{\pi}-} \frac{\partial^{2n+1}}{\partial \lambda^{2n+1}} 
		\e^{-2\sqrt{x} \cos(\lambda/2)}
		\ , \quad n\in\mathbb{N}_{0}.
\end{align}
\end{lemma}

\begin{proof} The relation \eqref{xPn KL inv}  can be rewritten like  
\begin{equation}\label{x Pn KL inv rel proof}
	x P_{n}(x) = \frac{1}{2^{2n+1}\pi} \lim_{\lambda\to \frac{\pi}{2}-} 
			\int_{0}^{\infty} \tau^{2n+2} 
			\cosh(\lambda \tau) K_{i{\tau}}(2\sqrt{x}) d\tau
			\ , \ n\in\mathbb{N}_{0},
\end{equation}
and, likewise, \eqref{sqrtx Pn tilde KL inv} can be restyled to 
\begin{equation}\label{sqrtx Pn tilde KL inv rel proof}
	\sqrt{x} \ \widetilde{P}_{n}(x) = \frac{1}{2^{2n}\pi} \lim_{\lambda\to \frac{\pi}{2}-} 
			\int_{0}^{\infty} \tau^{2n+1} 
			\sinh(\lambda {\tau}) K_{i{\tau}}(2\sqrt{x}) d\tau
			\ , \ n\in\mathbb{N}_{0}. 
\end{equation}

Considering that 
$$
	\cosh(\lambda \tau) \left(\frac{\tau}{2}\right)^{2n+2}  
	= 2^{-(2n+2)}\frac{\partial^{2n+2}}{\partial \lambda^{2n+2}} \cosh(\lambda \tau)
$$ and $$
	\sinh(\lambda \tau) \left(\frac{\tau}{2}\right)^{2n+1}  
	= 2^{-(2n+1)}\frac{\partial^{2n+1}}{\partial \lambda^{2n+1}} \cosh(\lambda \tau),
$$
it is reasonable to rewrite \eqref{x Pn KL inv rel proof} and \eqref{sqrtx Pn tilde KL inv rel proof} as follows 
\begin{align*}
	& 
	x P_{n}(x) 
	=  2^{-(2n+2)} \lim_{\lambda\to \pi/2{-}} \frac{2}{\pi} \frac{\partial^{2n+2}}{\partial \lambda^{2n+2}} 
		\int_{0}^\infty K_{i\tau}(2\sqrt{x})  \cosh(\lambda \tau) d\tau \ , \\
	&
	\sqrt{x} \ \widetilde{P}_{n}(x)
	=  2^{-(2n+1)} \lim_{\lambda\to \pi/2{-}} \frac{2}{\pi} \frac{\partial^{2n+2}}{\partial \lambda^{2n+2}} 
		\int_{0}^\infty K_{i\tau}(2\sqrt{x})  \cosh(\lambda \tau) d\tau \ ,
\end{align*}
motivated by the absolute and uniform convergence by $\lambda\in [0,\pi/2 -\epsilon ]$, for a small positive 
$\epsilon$. 
Meanwhile, from \eqref{Cosine Fourier K} and the inversion formula of the cosine Fourier transform, we deduce  \cite{YakubovichBook1996} 
\begin{equation}\label{int Kit cosh}
	\frac{2}{\pi} \int_{0}^\infty  \cosh(\lambda\tau) K_{i\tau} (2\sqrt{x}) d\tau =  \e^{-2\sqrt{x} \cos(\lambda/2)},
		\quad x>0, 
\end{equation}
which completes the proof. 
\end{proof}

\begin{remark} By taking $n=0$ in \eqref{KL s Ak Pn} and on account of  \eqref{KL s xn}, \eqref{central fact change basis1} and \eqref{xPn KL inv}, we deduce 
$$
	\frac{(-1)^m}{x} \mathcal{A}^m x^{k+1} 
	= \sum_{\nu=0}^{k} (-1)^{k+\nu}t_{E}(k+1,\nu+1) P_{m+\nu}(x) \ , \ 
	m,k\in\mathbb{N}_{0}.
$$
Similarly, upon the choice of $n=0$ in \eqref{KL c Ak Pn tilde} and considering \eqref{KL c xn}, \eqref{central fact change basis1} and \eqref{sqrtx Pn tilde KL inv}, we conclude 
$$
	\frac{(-1)^m}{\sqrt{x}} \mathcal{A}^m x^{k+\frac{1}{2}} 
	= \sum_{\nu=0}^{k} (-1)^{k+\nu}t_{O}(k,\nu) \widetilde{P}_{m+\nu}(x) \ , \ 
	m,k\in\mathbb{N}_{0}.
$$
\end{remark}

This latter result readily provides a generating function for each of the MPSs  $\{P_{n}(x)\}_{n\geqslant 0}$ and  $\{\widetilde{P}_{n}(x)\}_{n\geqslant 0}$, precisely 
\begin{equation} \label{GF 1 Pn}
	\frac{\partial^2}{\partial u^2}\frac{\e^{2\sqrt{x}\sin(u/2)}}{x} 
	= \sum_{n\geqslant0} P_{n}(x) \frac{u^{2n}}{(2n)!}
\end{equation}
and 
\begin{equation} \label{GF 1 Pn}
	\frac{\e^{2\sqrt{x}\sin(u/2)} }{\sqrt{x}} 
	= \sum_{n\geqslant0} \widetilde{P}_{n}(x) \frac{u^{2n+1}}{(2n+1)!}.
\end{equation}

The expressions for these generating functions could as well be attained from the developments made in \cite[Ch. 6, p.214]{Riordan2} and after a few steps of computations.

\subsection{Integral relation with Bernoulli and Euler numbers}\label{sec: genocchi euler numbers}

The {\it Genocchi numbers} are the coefficients in the Taylor series expansion of the function $\frac{2t}{\e^t +1}$ and they are connected to the Bernoulli numbers $\mathfrak{B}_{n}$ through $\mathfrak{G}_{n}=2(1-2^{n})\mathfrak{B}_{n}  , \ n\in\mathbb{N}_{0}$ \cite[(24.15.2)]{NIST}\cite[p.49]{Comtet} and admit the integral representation  \cite[Vol.I]{Bateman}: 
\begin{equation}\label{Bernoulli numbers integral}
	  \mathfrak{G}_{2n+2}
	 =(-1)^{n+1} (2n+2)\int_{0}^\infty \frac {\tau^{n} }{\sinh(\pi \sqrt{\tau})} d\tau
	 \ , \ n\in\mathbb{N}_{0}, 
\end{equation}
along with the one for Euler numbers  
\begin{equation}\label{Euler numbers integral}
	 E_{2n}
	 =(-1)^{n} 2^{2n}\int_{0}^\infty \frac {\tau^{n-\frac{1}{2} }}{\cosh(\pi \sqrt{\tau})}  d\tau
	 \ , \ n\in\mathbb{N}_{0}, 
\end{equation}
endows another integral meaning for the elements of the MPS $\{P_{n}\}_{n\geqslant 0}$, which in turn provides yet another identity between Genocchi and central factorial numbers. 

\begin{corollary}  The MPS $\{P_{n}\}_{n\geqslant 0}$ is connected with the Genochi numbers of even order via 
\begin{equation}\label{Bernoulli Pn} 
		 \frac{(-1)^{n+1} }{2n+2}\mathfrak{G}_{2n+2}
	=   \int_{0}^\infty \e^{-2\sqrt{x}} P_{n}(x) dx 
	\ , \quad n\in\mathbb{N}_{0}.
\end{equation}
whereas the MPS $\{\widetilde{P}_{n}\}_{n\geqslant 0}$ is connected with the Euler numbers of even through 
\begin{equation}\label{Euler Pn} 
		 (-1)^{n}2^{-2n} E_{2n}
	=   \int_{0}^\infty \e^{-2\sqrt{x}} \widetilde{P}_{n}(x) \frac{dx}{\sqrt{x}} 
	\ , \quad n\in\mathbb{N}_{0}.
\end{equation}
\end{corollary}

\begin{proof} The equality \eqref{KL s Pn} can be rewritten like  
$$
	\frac{\tau^{n } }{\sinh(\pi\sqrt{\tau})} 
	= \frac{2}{\pi\sqrt{\tau}} \int_{0}^\infty K_{2i\sqrt{\tau}}(2\sqrt{x}) P_{n}(x) dx 
	\ , \ n\in\mathbb{N}_{0},
$$
and, in turn, \eqref{KL Pntilde} can be restyled into 
$$
	\frac{\tau^{n-\frac{1}{2}}}{\cosh(\pi\sqrt{\tau})}
	=\frac{2}{\pi\sqrt{\tau}}\int_{0}^{\infty} K_{2i\sqrt{\tau}}(2\sqrt{x}) \widetilde{P}_{n}(x)\frac{dx}{\sqrt{x}}
	 \ , \ n\in\mathbb{N}_{0}. 
$$
We integrate both sides of each of the latter equations over $\mathbb{R}_{+}$ by $\tau$ and we interchange the order of integration in the left hand-side of the equalities, according to Fubini's theorem  on the grounds of the  inequality \cite{YakubovichBook1996}
\begin{equation}\label{ineq Kitau}
	\left| K_{2i\sqrt{\tau}}(2\sqrt{x}) \right| \leqslant   \e^{-\delta \sqrt{\tau}} K_{0}(2\sqrt{x}\cos \delta) , \quad x>0, \ \tau >0 , \ \delta \in (0,\pi/2).
\end{equation} 
By virtue of the relation 
$$
	\int_{0}^\infty K_{2i\sqrt{\tau}}(2\sqrt{x}) \frac{ d\tau}{\sqrt{\tau}} = \frac{\pi}{2} \e^{-2\sqrt{x}}
$$
together with \eqref{Bernoulli numbers integral} and \eqref{Euler numbers integral}, we respectively achieve the identities  \eqref{Bernoulli Pn} and \eqref{Euler Pn}. 
\end{proof}

The relation \eqref{Bernoulli Pn} also provide another identity between the Genocchi and central factorial numbers of even order, insofar as, the input of \eqref{Pn explicit exp} in \eqref{Bernoulli Pn} yields 
$$
	\mathfrak{G}_{2n+2}=(2n+2)\sum_{\nu=0}^{n} (-1)^{\nu+1} T_{E}(n+1,\nu+1) 2^{-2\nu-1} 
				(2\nu+1)! 
				\ , \ n\in\mathbb{N}_{0}. 
$$
Likewise, recalling \eqref{Pntilde explicit exp} the integral relation  \eqref{Euler Pn} provides 
$$
	2^{-2n}E_{2n}=\sum_{\nu=0}^{n} (-1)^{\nu+1} T_{O}(n,\nu) 2^{-2\nu} 
				(2\nu)! 
				\ , \ n\in\mathbb{N}_{0}. 
$$

\subsection{Structural relations}\label{subsec: rec rel Pn Pntilde}

In order to figure out the structural algebraic relation for the MPS $\{P_{n}\}_{n\geqslant 0}$ and $\{\widetilde{P}_{n}\}_{n\geqslant 0}$ we will proceed to the computation of the $KL_{s}$-transform of $xP_{n}(x)$ along with the $KL_{c}$-transform of $x\widetilde{P}_{n}(x)$. 
At a first glance, by making use of the explicit expression  \eqref{Pn explicit exp} for the polynomials $P_{n}$ and the linearity of the KL transform, we straightforwardly obtain 
\begin{eqnarray*}
KL_{s}[x P_{n}(x)](\tau)
       &=& \frac{2\sinh(\pi\sqrt{\tau})}{\pi\sqrt{\tau}} \int_{0}^\infty K_{2i\sqrt{\tau}}(2\sqrt{x}) x P_{n}(x) dx \\
       &=& \sum_{\nu=0}^{n}(-1)^{n+\nu}T_{E}(n+1,\nu+1)  \frac{2\sinh(\pi\sqrt{\tau})}{\pi\sqrt{\tau}} \int_{0}^\infty K_{2i\sqrt{\tau}}(2\sqrt{x}) \ x^{\nu+1} dx \ , \ n\in \mathbb{N}_{0}, \\
\end{eqnarray*}
and thereby 
\begin{equation} \label{KLxPn Factorials}
KL_{s}[xP_{n}(x)](\tau)= \sum_{\nu=0}^{n}(-1)^{n+\nu}T_{E}(n+1,\nu+1)
               \prod_{\sigma=1}^{\nu+1}\left(\sigma^2 + \tau\right)\ .
\end{equation}
In a similar manner, we deduce, in view of \eqref{KL c xn} and \eqref{Pntilde explicit exp}, the identity 
\begin{equation} \label{KL c xPn tilde Factorials}
KL_{c}[x\widetilde{P}_{n}(x)](\tau)
		= \sum_{\nu=0}^{n}(-1)^{n+\nu}T_{O}(n,\nu)
               	\prod_{\sigma=0}^{\nu}\left(\left(\tfrac{1}{2}+\sigma\right)^2 + \tau\right) \ , \ 
		n\in\mathbb{N}_{0}. 
\end{equation}

Despite the simple appearance of the just obtained expressions, it seems considerably complicate to deduce from them a recursive-structural expression for the two MPSs $\{P_{n}\}_{n\geqslant 0}$ and $\{\widetilde{P}_{n}\}_{n\geqslant 0}$. We will indeed succeed in achieving this goal after some manipulations with the KL-transforms which will result in identities with the Euler polynomials. Yet comprised in this section is the determination of the connection coefficients between $\{P_{n}\}_{n\geqslant 0}$ and $\{\widetilde{P}_{n}\}_{n\geqslant 0}$

Incidentally, these procedures will generate some identities between the Euler polynomials and the central factorial polynomials, Euler numbers and the central factorial numbers, which will receive a separate attention during \S \ref{Sec: Number identities}.

On the other hand, a suitable change on the order of integration sustained by \eqref{KL xf Euler} permits to deduce a connection between the latter expression and the Euler polynomials, $E_{n}(x)$, commonly defined by  \cite[(23.1.1)]{Abramowitz}\cite{Comtet,NIST} via 
$$
	\frac{2 \e^{tx}}{\e^t +1} = \sum_{n\geqslant 0} E_{n}(x) \frac{t^n}{n!} \ .
$$ 
They are a key ingredient to further our ends, specially their explicit expression in terms of the monomials \cite[(24.4.14)]{NIST}
\begin{equation} \label{EulerExpansion}
	E_{n}(x) = \frac{1}{n+1} \sum_{\nu=0}^{n} \binom{n+1}{\nu} \mathfrak{G}_{n+1-\nu} x^{\nu}
	\ , \ n\in\mathbb{N}_{0}, 
\end{equation}
where $\mathfrak{G}_{n}$ represent the aforementioned Genocchi numbers (see \S\ref{sec: genocchi euler numbers}).

\begin{theorem}\label{Theo: phi and psi lambda} Let $f(x)$, $\ x\in\mathbb{R}_{+},$ be a continuos function of bounded variation  such that \linebreak ${f(x)\in L_{1}(\mathbb{R}_{+}, K_{0}(2\alpha\sqrt{x})dx)}$, $0<\alpha<1$ and  
$$
	\phi_{\lambda}(x) = \frac{2}{\pi} \int_{0}^{\infty} \sqrt{\mu} \cosh(\lambda \sqrt{\mu}) 
					K_{2i\sqrt{\mu}}(2\sqrt{x})  KL_{s}[f](\mu) d\mu\ , \ x>0 . 
$$ 
and 
$$
	\psi_{\lambda}(x) = \frac{2\sqrt{x}}{\pi} \int_{0}^{\infty}  \sinh(\lambda \sqrt{\mu}) 
					K_{2i\sqrt{\mu}}(2\sqrt{x})  KL_{c}[f](\mu) d\mu\ , \ x>0 . 
$$ 
If the functions $ \Phi_{\tau}(\lambda)= KL_{s}[\phi_{\lambda}](\tau) $ and $ \Psi_{\tau}(\lambda)= KL_{c}[\psi_{\lambda}](\tau) $,  $\lambda\in[0,\pi]$, 
are continuous at the point $\lambda=\pi$ for each $\tau\in\mathbb{R}_{+}$, then 
\begin{equation}\label{KL xf Euler} 
		\ds \lim_{\lambda \to \pi-}KL_{s}[\phi_{\lambda}](\tau)
		=KL_{s}[xf(x)](\tau) 
		=\ds  \frac{2}{\sqrt{\tau}} 
			\int_{0}^{\infty}
			 \frac{\sqrt{\mu} \sinh(\pi\sqrt{\tau})\cosh(\pi \sqrt{\mu}) (\tau - \mu)}{\cosh\left(2\pi\sqrt{\tau} \right)
		-\cosh\left(2\pi\sqrt{\mu} \right)} KL_{s}[f](\mu) d\mu 
\end{equation}
while 
\begin{equation}\label{KL c xf Euler} 
		\ds \lim_{\lambda \to \pi-}KL_{c}[\psi_{\lambda}](\tau)
		=KL_{c}[{x}f(x)](\tau) 
		=\ds  {2 } 
			\int_{0}^{\infty}
			 \frac{ \cosh(\pi\sqrt{\tau})\sinh(\pi \sqrt{\mu}) (\tau - \mu)}{\cosh\left(2\pi\sqrt{\tau} \right)
		-\cosh\left(2\pi\sqrt{\mu} \right)} KL_{c}[f](\mu) d\mu \ . 
\end{equation}
\end{theorem}

\begin{proof} Within the framework \eqref{KL s directa} the $KL_{s}$-transform of $\phi_{\lambda}$ is given by 
$$
	KL_{s}[\phi_{\lambda}](\tau) = \frac{2\sinh(\pi\sqrt{\tau})}{\pi \sqrt{\tau}}
			\int_{0}^{\infty} K_{2i\sqrt{\tau}}(2\sqrt{x}) \phi_{\lambda}(x) dx .
$$
Likewise, recalling \eqref{KL c directa}, the $KL_{c}$-transform of $\psi_{\lambda}$ becomes 
$$
	KL_{c}[\psi_{\lambda}](\tau) = \frac{2\cosh(\pi\sqrt{\tau})}{\pi }
			\int_{0}^{\infty} K_{2i\sqrt{\tau}}(2\sqrt{x}) \psi_{\lambda}(x) \frac{dx}{\sqrt{x}} .
$$

Now, we interchange the order of integration on the right-hand side of the latter equality according to Fubini's theorem and on account of the inequality \eqref{ineq Kitau}, 
the condition $f\in L_{1}(\mathbb{R}_{+}, K_{0}(2\alpha\sqrt{x})dx)$, $0<\alpha<1$ and the following estimate 
\begin{eqnarray*}
	&&\int_{0}^{\infty}\int_{0}^{\infty}\int_{0}^{\infty} 
			\sqrt{\mu} \cosh(\lambda \sqrt{\mu}) 
			\left| K_{2i\sqrt{\tau}}(2\sqrt{x})
			K_{2i\sqrt{\mu}}(2\sqrt{x})  K_{2i\sqrt{\mu}}(2\sqrt{y}) f(y) \right| dy d\mu dx \\
	&&\leqslant
		\e^{-\delta \sqrt{\tau}}
		\int_{0}^{\infty}\int_{0}^{\infty}\int_{0}^{\infty} 
			\sqrt{\mu} \cosh(\lambda \sqrt{\mu})  \e^{-(\delta_{1}+\delta_{2}) \sqrt{\mu}}\\
	&& \times 
			K_{0}(2\sqrt{x}\cos \delta)
			K_{0}(2\sqrt{x}\cos \delta_{1})K_{0}(2\sqrt{y}\cos \delta_{2}) \left| f(y) \right| dy d\mu dx\ 
			< \infty , 
\end{eqnarray*}
where $\delta_{i} \in (0,\pi/2)$ for $i=1,2$, and $\delta_{1}+\delta_{2} >\lambda$. As a result, 
\begin{equation}\label{KL phi lambda}
	KL_{s}[\phi_{\lambda}(x)](\tau) = \frac{4\sinh(\pi\sqrt{\tau})}{\pi^{2} \sqrt{\tau}} 
			\int_{0}^{\infty}\sqrt{\mu} \cosh(\lambda \sqrt{\mu}) KL[f(x)](\mu) \mathcal{K}(\tau,\mu) d\mu
\end{equation}
while 
\begin{equation}\label{KL c psi lambda}
	KL_{c}[\psi_{\lambda}(x)](\tau) = \frac{4\cosh(\pi\sqrt{\tau})}{\pi^{2} } 
			\int_{0}^{\infty}  \sinh(\lambda \sqrt{\mu}) KL_{c}[f(x)](\mu) \mathcal{K}(\tau,\mu) d\mu
\end{equation}
where the kernel $\ds\mathcal{K}(\tau,\mu)=\int_{0}^\infty K_{2i\sqrt{\tau}}(2\sqrt{x})  K_{2i\sqrt{\mu}}(2\sqrt{x})   dx$  is indeed an elementary function, which can be calculated via relation (2.16.33.2) in \cite[Vol.2]{PrudnikovMarichev}: 
\begin{eqnarray*}
	\mathcal{K}(\tau,\mu)=\int_{0}^\infty K_{2i\sqrt{\tau}}(2\sqrt{x})  K_{2i\sqrt{\mu}}(2\sqrt{x})  dx
	&=& \frac{\pi^2}{2} \frac{\tau - \mu}{\cosh\left(2\pi\sqrt{\tau} \right)
		-\cosh\left(2\pi\sqrt{\mu} \right)} \ .
\end{eqnarray*}
Furthermore,  the passage to the limit  $\lambda \to \pi-$ under the integral sign on the right-hand side either of \eqref{KL phi lambda} or of \eqref{KL c psi lambda}
 can be justified similarly due to the absolutely and uniform convergence via the Weierstrass test. 
\end{proof}

The latter result is of the utmost importance to deduce a simple expression for the $KL_{\{s,c\}}$-transforms of $x {P}_{n}(x)$ and $x\widetilde{P}_{n}(x)$, which, in turn, provide the structural-recursive relation for the corresponding sequences.

\begin{theorem}\label{Theo: KL de xPn} The $KL_{s}$ transform of $xP_{n}(x)$ is given by 
\begin{equation} \label{KLxPn Euler even}   
KL_{s}[xP_{n}(x)](\tau)=   \frac{ (-1)^{n}}{ i\sqrt{\tau}} 
		\Bigg\{ \tau  E_{2n+2}(i\sqrt{\tau}) 
		 + E_{2n+4}(i\sqrt{\tau}) \Bigg\}  
		 =  \sum_{k=0}^{n+1}\binom{2n+2}{2k} 
		\frac{ (n+k+2) \mathfrak{G}_{2n-2k+4}  }
			{(2 k+1) (n-k+2)} (-1)^{n+k}  {\tau}^{k}
\end{equation}  
while the $KL_{c}$ transform of $x\widetilde{P}_{n}(x)$ can be expressed as follows  
\begin{equation} \label{KLxPn Euler odd}   
KL_{c}[x\widetilde{P}_{n}(x)](\tau)=   (-1)^{n}\Bigg\{\tau 
		E_{2n+1}(i\sqrt{\tau} )  
		+ E_{2n+3}(i\sqrt{\tau} ) \Bigg\} 
		= \sum_{k=0}^{n+1} \binom{2 n+2}{2 k} \frac{(n+k+1) \mathfrak{G}_{2n-2k+4} }
		 {2 (n+1) (n-k+2)} (-1)^{n+k}  {\tau}^{k}
\end{equation}  
where $E_{n}(\cdot)$ represent the Euler polynomials and $\mathfrak{G}_{n}$ the {\it Genocchi numbers}. 
Consequently, MPS $\{P_{n}\}_{n\geqslant 0}$ fulfills the structure relation 
\begin{align} \label{structure relation Pn}
	& P_{n+2}(x) =  \big(x-2(n+2)^{2}\big) P_{n+1}(x)  - \sum_{k=0}^{n} 
		 \binom{2n+4}{2k} \frac{(-1)^{n+k+1}  (n+k+4) \mathfrak{G}_{2n-2k+6}  }
			{(2 k+1) (n-k+4)}  P_{k}(x)\\
	& \label{structure relation Pn tilde}
	\widetilde{P}_{n+2}(x) =  \big(x-\left(n+\tfrac{3}{2}\right)^{2}\big) \widetilde{P}_{n+1}(x)  
		- \sum_{k=0}^{n} 
		  \binom{2 n+4}{2 k} \frac{(-1)^{n+k+1} (n+k+2) \mathfrak{G}_{2n-2k+6} }
		 {2 (n+2) (n-k+3)} 
		  \widetilde{P}_{k}(x)
\end{align}
for $n\in\mathbb{N}_{0}$, with $P_{0}(x)=1$ and $P_{1}(x)=x-1$.
\end{theorem}

\begin{proof} The even order Euler polynomials admit the integral representation  \cite[p. 43, Vol.I]{Bateman}\cite[(24.7.9)]{NIST}
\begin{equation}\label{EulerPolys even}
	E_{2n+2}(x ) = 4 (-1)^{n+1}  
			\int _{0}^{\infty} \frac{\sin(\pi x) \cosh(\pi \mu )}{\cosh(2\pi \mu) - \cos(2 \pi x)} 
			 \mu^{2n+2} d\mu
			\ , \ n\in\mathbb{N}_{0}, 
\end{equation}
whereas those of odd order can be represented 
\begin{equation}\label{EulerPolys odd}
	E_{2n+1}(x ) = 4 (-1)^{n+1}  \int _{0}^{\infty} \frac{\cos(\pi x) \sinh(\pi \mu )}
			{\cosh(2\pi \mu) - \cos(2 \pi x)} 
			 \mu^{2n+1} d\mu
			\ , \ n\in\mathbb{N}_{0}, 
\end{equation}
both of them valid whenever $0<\Re(x)<1$. Thus, by taking $x=i\sqrt{\tau} +\epsilon$ and upon the change of variable $\mu\to\sqrt{\mu}$, it is valid the representation 
\begin{equation}\label{EulerPolys i tau even}
	 E_{2n+2}(i\sqrt{\tau} +\epsilon) 
	 = 2 i (-1)^{n+1}  \int _{0}^{\infty} \frac{ \sinh( \pi  (\sqrt{\tau} -i\epsilon)) \cosh(\pi \sqrt{\mu})}
	 		{\cosh(2\pi \sqrt{\mu}) - \cosh(2{\pi}  (\sqrt{\tau}-i\epsilon))} 
			\mu^{n+\frac{1}{2}} d\mu
			\ , \ n\in\mathbb{N}_{0}. 
\end{equation}
and 
\begin{equation}\label{EulerPolys i tau odd}
	 E_{2n+1}(i\sqrt{\tau} +\epsilon) 
	 = 2 (-1)^{n+1}  \int _{0}^{\infty} \frac{ \cosh( \pi  (\sqrt{\tau} -i\epsilon)) \sinh(\pi \sqrt{\mu})}
	 		{\cosh(2\pi \sqrt{\mu}) - \cosh(2{\pi}  (\sqrt{\tau}-i\epsilon))} 
			\mu^{n } d\mu
			\ , \ n\in\mathbb{N}_{0}. 
\end{equation}
In the light of Theorem \ref{Theo: phi and psi lambda}, the replacement $f(x)=P_{n}(x)$  on the relation \eqref{KL xf Euler} and on account of \eqref{KL s Pn} provides 
$$
	KL_{s}[xP_{n}(x)](\tau) = \frac{2\sinh (\pi \sqrt{\tau} )}{\sqrt{\tau}} 
				\int_{0}^{\infty} \mu^{n+\frac{1}{2}}
				  \frac{(\tau - \mu)  
				\cosh(\pi \sqrt{\mu})  }{\cosh\left(2\pi\sqrt{\tau} \right)
		-\cosh\left(2\pi\sqrt{\mu} \right)}  d\mu
		\ , \ n\in\mathbb{N}_{0}.
$$ 
while the substitution of $f(x)=\widetilde{P}_{n}(x)$ on  \eqref{KL c xf Euler} and on account of \eqref{KL Pntilde} provides  
$$
	KL_{c}[x\widetilde{P}_{n}(x)](\tau) ={2\cosh (\pi \sqrt{\tau} )}
				\int_{0}^{\infty} \mu^{n }
				  \frac{(\tau - \mu)  
				\sinh(\pi \sqrt{\mu})  }{\cosh\left(2\pi\sqrt{\tau} \right)
		-\cosh\left(2\pi\sqrt{\mu} \right)}  d\mu
		\ , \ n\in\mathbb{N}_{0}.
$$ 
The accomplishment of the proof depends on the vital argument that is now our target: 
\begin{equation}\label{KL Pns epsilon}
	KL_{s}[xP_{n}(x)](\tau) =
			\frac{1}{\sqrt{\tau}}
			 \lim_{\varepsilon\to0+}
				\Phi_{n}(\varepsilon,\tau)
\quad 
\text{along with }\quad 
	KL_{c}[x\widetilde{P}_{n}(x)](\tau) =
			 \lim_{\varepsilon\to0+}
				\Psi_{n}(\varepsilon,\tau)
		\ , \ n\in\mathbb{N}_{0},
\end{equation}
where 
\begin{align*}
	\Phi_{n}(\varepsilon,\tau) 
	& =2\int_{0}^{\infty} \mu^{n+\frac{1}{2}}
				  \frac{((\sqrt{\tau}-i\varepsilon)^2 - \mu)  
				  \sinh (\pi (\sqrt{\tau}-i\varepsilon) )
				\cosh(\pi \sqrt{\mu})  }
				{\cosh\left(2\pi(\sqrt{\tau}-i\varepsilon) \right) -\cosh\left(2\pi\sqrt{\mu} 
				\right)}  d\mu \\
	& =   \int_{0}^{\infty} \mu^{n+\frac{1}{2}}\left\{ 
				  \frac{-(\sqrt{\mu}+\sqrt{\tau}-i\varepsilon)
				  (\sqrt{\mu}-\sqrt{\tau}+i\varepsilon)   }
				{\sinh (\pi (\sqrt{\mu}+\sqrt{\tau}-i\varepsilon) )} 
				+ \frac{(\sqrt{\mu}+\sqrt{\tau}+i\varepsilon)
				  (\sqrt{\mu}-\sqrt{\tau}+i\varepsilon)   }
				{\sinh (\pi (\sqrt{\mu}-\sqrt{\tau}+i\varepsilon) )} 
				\right\} d\mu 
\end{align*}
for all $n\in\mathbb{N}_{0}$, and  
\begin{align*}
	\Psi_{n}(\varepsilon,\tau) 
	& =2\int_{0}^{\infty} \mu^{n}
				  \frac{((\sqrt{\tau}-i\varepsilon)^2 - \mu)  
				  \cosh (\pi (\sqrt{\tau}-i\varepsilon) )
				\sinh(\pi \sqrt{\mu})  }
				{\cosh\left(2\pi(\sqrt{\tau}-i\varepsilon) \right) -\cosh\left(2\pi\sqrt{\mu} 
				\right)}  d\mu\\
	& =   \int_{0}^{\infty} \mu^{n}\left\{ 
				  \frac{(\sqrt{\mu}+\sqrt{\tau}-i\varepsilon)
				  (\sqrt{\mu}-\sqrt{\tau}+i\varepsilon)   }
				{\sinh (\pi (\sqrt{\mu}+\sqrt{\tau}-i\varepsilon) )} 
				+ \frac{(\sqrt{\mu}+\sqrt{\tau}+i\varepsilon)
				  (\sqrt{\mu}-\sqrt{\tau}+i\varepsilon)   }
				{\sinh (\pi (\sqrt{\mu}-\sqrt{\tau}+i\varepsilon) )} 
				\right\} d\mu\ , \ n\in\mathbb{N}_{0}.
\end{align*} 
Indeed, it suffices to prove that for each $\tau >0$ and for sufficiently small $\delta>\varepsilon\geqslant0$, the integral   
$$\begin{array}{l}
	\ds \int_{0}^{\infty} \mu^{n+\frac{1}{2}}
	 \frac{(\sqrt{\mu}+\sqrt{\tau}+i\varepsilon)
				  (\sqrt{\mu}-\sqrt{\tau}+i\varepsilon)   }
				{\sinh (\pi (\sqrt{\mu}-\sqrt{\tau}+i\varepsilon) )} 
				 d\mu \\
	\ds =\left\{ \int_{|\sqrt{\mu}-\sqrt{\tau}|\geqslant \delta >0} d\mu
	+\int_{|\sqrt{\mu}-\sqrt{\tau}|<\delta} d\mu\right\}
	\left( \mu^{n+\frac{1}{2}}
	 \frac{(\sqrt{\mu}+\sqrt{\tau}+i\varepsilon)
				  (\sqrt{\mu}-\sqrt{\tau}+i\varepsilon)   }
				{\sinh (\pi (\sqrt{\mu}-\sqrt{\tau}+i\varepsilon) )} \right) \ , \ n\in\mathbb{N}_{0}, 
\end{array}$$
converges uniformly by $\varepsilon\in[0,\delta)$, which can be concluded regarding that 
$$
	\int_{|\sqrt{\mu}-\sqrt{\tau}|\geqslant \delta >0}
	\left|  \mu^{n+\frac{1}{2}}
	 \frac{(\sqrt{\mu}+\sqrt{\tau}+i\varepsilon)
				  (\sqrt{\mu}-\sqrt{\tau}+i\varepsilon)   }
				{\sinh (\pi (\sqrt{\mu}-\sqrt{\tau}+i\varepsilon) )}   \right| d\mu
	\leqslant C_{\tau} \int_{|y|\geqslant \delta>0}
	\frac{\left| (y+\sqrt{\tau})^{2n+1} y\right|}{\left| \sinh\pi y\right|} dy <\infty
$$
uniformly by $0\leqslant\varepsilon<1$, whereas 
$$
	\int_{|\sqrt{\mu}-\sqrt{\tau}|\leqslant \delta }
	\left|  \mu^{n}
	 \frac{(\sqrt{\mu}+\sqrt{\tau}+i\varepsilon)
				  (\sqrt{\mu}-\sqrt{\tau}+i\varepsilon)   }
				{\sinh (\pi (\sqrt{\mu}-\sqrt{\tau}+i\varepsilon) )}   \right| d\mu
	\leqslant C_{\tau} \int_{|y|\leqslant \delta}
	\frac{\left| y+i\varepsilon\right|}{\left| \sinh\pi (y+i\varepsilon)\right|} dy <C_{\tau}2\delta <\infty
$$
uniformly by $0\leqslant\varepsilon<\delta$. Thus, the integrals $\Phi_{n}(\varepsilon,\tau) $ and $\Psi_{n}(\varepsilon,\tau) $ converge uniformly by $\varepsilon$, according to the Weierstrass test, which legitimates the passage of the limit as $\varepsilon \to 0+$ under the integral signs so that we have motivated the validity of \eqref{KL Pns epsilon}. 

On the other hand, in the light of \eqref{EulerPolys i tau even}-\eqref{EulerPolys i tau odd}, $\Phi_{n}(\varepsilon,\tau) $ and $\Psi_{n}(\varepsilon,\tau) $ can be written in terms of even order Euler polynomials 
$$	\Phi_{n}(\varepsilon,\tau) 
	= {(-1)^{n}i}  \ 
		\Bigg\{ \left( \sqrt{\tau}-i\varepsilon \right)^2  
		E_{2n+2}(i\sqrt{\tau} +\epsilon)  
		+ E_{2n+4}(i\sqrt{\tau} +\epsilon) \Bigg\}
$$
while 
$$	\Psi_{n}(\varepsilon,\tau) 
	= (-1)^{n}\Bigg\{ \left( \sqrt{\tau}-i\varepsilon \right)^2  
		E_{2n+1}(i\sqrt{\tau} +\epsilon)  
		+ E_{2n+3}(i\sqrt{\tau} +\epsilon) \Bigg\} \ ,
$$
which, together with \eqref{KL Pns epsilon}, ensure the first equalities on \eqref{KLxPn Euler even} and \eqref{KLxPn Euler odd}. Meanwhile, the expansion for the Euler polynomials \eqref{EulerExpansion} enables 
\begin{equation}\label{Euleres explic even}
	\frac{ (-1)^{n}}{ i\sqrt{\tau}} 
		\Bigg\{ \tau  E_{2n+2}(i\sqrt{\tau}) 
		 + E_{2n+4}(i\sqrt{\tau}) \Bigg\} 	
	  =  \sum_{k=0}^{n+1}\binom{2n+2}{2k} 
	\frac{ (n+k+2) \mathfrak{G}_{2n-2k+4}  }
			{(2 k+1) (n-k+2)} (-1)^{n+k+1} \tau^{k}
\end{equation}
and 
\begin{equation}\label{Euleres explic odd}
	 (-1)^{n} 
		\Bigg\{ \tau  E_{2n+1}(i\sqrt{\tau}) 
		 + E_{2n+3}(i\sqrt{\tau}) \Bigg\} 	
	  =    \sum_{k=0}^{n+1} \binom{2 n+2}{2 k} \frac{(k+n+1) \mathfrak{G}_{2n-2k+4} }
		 {2 (n+1) (n-k+2)} (-1)^{n+k}  {\tau}^{k}
\end{equation}
whence, the equalities between the middle and last member on both \eqref{KLxPn Euler even} and \eqref{KLxPn Euler odd} hold. 

Moreover, bearing in mind the injectivity and linearity of the KL transform (and in particular $KL_{s}$ and $KL_{c}$) and on account of \eqref{KL s Pn}, the identity \eqref{Euleres explic even} gives rise to 
$$
	x P_{n}(x)
	= \sum_{k=0}^{n+1}\binom{2n+2}{2k} 
	\frac{ (n+k+2) \mathfrak{G}_{2n-2k+4}  }
			{(2 k+1) (n-k+2)} (-1)^{n+k} P_{k}(x)
$$
and, likewise, \eqref{Euleres explic odd} provides  
$$
	x \widetilde{P}_{n}(x)
	= \sum_{k=0}^{n+1} \binom{2 n+2}{2 k} \frac{(n+k+1) \mathfrak{G}_{2n-2k+4} }
		 {2 (n+1) (n-k+2)} (-1)^{n+k} \widetilde{P}_{k}(x) \ ,
$$
providing \eqref{structure relation Pn tilde}, because of $\mathfrak{G}_{2}=-1=-\mathfrak{G}_{4}$. 
\end{proof}

\begin{remark} \label{rem: nonortho} Noteworthy to notice is the fact that the relation \eqref{structure relation Pn} shows that $\{P_{n}\}_{n\geqslant 0}$ or $\{ \widetilde{P}_{n}(x)\}_{n\geqslant 0}$ are not orthogonal (with respect to a $L_{2}$-inner product), insofar as they do not fulfill a second-order recursive relation. Or, even more generally,  neither the MPS $\{P_{n}\}_{n\geqslant 0}$ nor $\{ \widetilde{P}_{n}(x)\}_{n\geqslant 0}$ can be $d$-orthogonal, because it does not fulfill a recursive relation whose order is not independent from the order of the elements. For more details regarding the $d$-orthogonality we refer to \cite{MaroniDortho,Iseghem}. 

Notwithstanding these negative results, we cannot discard the possibility of the existence of an inner product with respect to which it would be possible to give some orthogonal sense to each of the MPSs  $\{P_{n}\}_{n\geqslant 0}$ or $\{ \widetilde{P}_{n}(x)\}_{n\geqslant 0}$. One possibility would be the Sobolev orthogonality, but we defer the discussion for a further work, leaving this as an open problem. 
\end{remark}

\subsection{Connection coefficients between $\{ P_{n}(x)\}_{n\geqslant 0}$ and $\{ \widetilde{P}_{n}(x)\}_{n\geqslant 0}$}

The procedure taken within this section to obtain the connection coefficients between $\{ P_{n}(x)\}_{n\geqslant 0}$ and $\{ \widetilde{P}_{n}(x)\}_{n\geqslant 0}$ is much similar to the one taken for the obtention of the structural recursive relation for each of the sequences. The idea is to obtain integral relations between the two KL modified transforms of a certain function $f$ satisfying the required conditions.

Motivated by similar arguments as the ones evoked in the proof of Theorem \ref{Theo: phi and psi lambda} and under the same assumptions over the function $f(x)$, then, bearing in mind the identity (see (2.16.33.2) in \cite[Vol.2]{PrudnikovMarichev}) 
$$
	\int_{0}^\infty K_{2i\sqrt{\tau}}(2\sqrt{x})  K_{2i\sqrt{\mu}}(2\sqrt{x})  \frac{dx}{\sqrt{x}}
	= \frac{\pi^2}{2} \frac{1}{\cosh(2\pi\sqrt{\mu}) + \cosh(2\pi\sqrt{\tau})} , 
$$ 
the application of the $KL_{c}$-transform, defined in \eqref{KL c directa}, over \eqref{KL s inverse} leads to 
\begin{equation}\label{KLc x KLs f}
	KL_{c}[x f(x)](\tau) 
	={2\cosh(\pi\sqrt{\tau})} \int_{0}^{\infty} \frac{\sqrt{\mu} \cosh(\pi\sqrt{\mu}) }
	{\cosh(2\pi\sqrt{\mu}) + \cosh(2\pi\sqrt{\tau})} KL_{s}[ f(x)](\tau) d\mu
\end{equation}
whereas the action of the $KL_{s}$-transform, defined in \eqref{KL s directa}, over \eqref{KL c inverse} divided by $\sqrt{x}$ results in 
\begin{equation}\label{KLs of KLc f}
	KL_{s}[f(x)](\tau)
	=  \frac{2\sinh(\pi\sqrt{\tau})}{\sqrt{\tau}} \int_{0}^{\infty} \frac{   \sinh(\pi\sqrt{\mu}) }
	{\cosh(2\pi\sqrt{\mu}) + \cosh(2\pi\sqrt{\tau})} KL_{c}[f](\mu) d\mu \ .
\end{equation}
Both of them are valid as long as $\phi_{\lambda}(x)$ and $\frac{1}{x} \psi_{\lambda}(x)$, with $\lambda\in[0,\pi]$, are continuous functions at the point $\lambda=\pi$ for each $\tau\in\mathbb{R}_{+}$.

On the other hand, mimicking the framework of Theorem \ref{Theo: KL de xPn}, then upon the choice of $f(x)=P_{n}(x)$ on the  relation \eqref{KLc x KLs f} and, afterwards, by evoking \eqref{EulerPolys even} we come out with the relation 
\begin{equation}\label{Eulereven KLc xPn}
	(-1)^{n+1} E_{2n+2}\left( i\sqrt{\tau}+\frac{1}{2}\right)
	= 2\int_{0}^\infty \frac{\cosh(\pi\sqrt{\tau})\cosh(\pi\sqrt{\mu})}{\cosh(2\pi\sqrt{\tau})+\cosh(2\pi\sqrt{\mu})}\mu^{n+\frac{1}{2}} d\mu
	=KL_{c}[xP_{n}(x)](\tau) \ .
\end{equation}

Likewise, the relation \eqref{KLs of KLc f} upon the replacement of $f(x)=\widetilde{P}_{n}(x)$ together with \eqref{EulerPolys odd} gives rise to 
\begin{equation}\label{Eulerodd KLs Pntilde}
	\frac{(-1)^{n} }{i\sqrt{\tau}} E_{2n+1}\left( i\sqrt{\tau}+\frac{1}{2}\right)
	= \frac{2}{\sqrt{\tau}}\int_{0}^\infty \frac{\sinh(\pi\sqrt{\tau})\sinh(\pi\sqrt{\mu})}{\cosh(2\pi\sqrt{\tau})+\cosh(2\pi\sqrt{\mu})}\mu^{n+\frac{1}{2}} d\mu
	=KL_{s}[\widetilde{P}_{n}(x)](\tau) \ .
\end{equation}

Since 
$\left| \frac{\partial^{2n+2}}{\partial \lambda^{2n+2}} 
	\e^{-2\sqrt{x} \cos(\lambda/2)} \right| \leqslant C_{n} x^n$ for $x>0$ and $\lambda\in[0,\pi/2)$, where $C_{n}$ depends exclusively on $n\in\mathbb{N}_{0}$, the fact that 
$ \left| K_{y+1}(2\sqrt{x})\right|\leqslant K_{\Re (y)+1}(2\sqrt{x})$  and 
$$\int_{0}^\infty K_{\Re (y)+1}(2\sqrt{x}) x^n dx < +\infty , 
$$ 
gives grounds for the passage of the limit under the outer integral in
\begin{align*}
	& E_{2n}\left(-\frac{y}{2}\right)  =  \frac{2}{\pi}(-1)^{n+1} \sin\left(\frac{\pi y}{2}\right) 
		\lim_{\lambda\to \pi^{-}} \frac{\partial^{2n+2}}{\partial \lambda^{2n+2}} 
		 \int_{0}^\infty  \e^{-2\sqrt{x} \cos(\lambda/2)} K_{y+1}(2\sqrt{x}) \frac{dx}{\sqrt{x}}.
\end{align*}
Appealing to \eqref{repre Pn in pi over 2} and as long as $-2<\Re y<0$, we derive
\begin{align} \label{eq Euler y even}
	E_{2n}\left(-\frac{y}{2}\right) 
	= \frac{2}{\pi} (-1)^{n+1} \sin(\pi y/2)  \int_{0}^\infty \sqrt{x} P_{n-1}(x) \  K_{y+1}(2\sqrt{x}) dx
	\ , \ n\in\mathbb{N},
\end{align} 
which yields  \eqref{Eulereven KLc xPn} when $y=-1+2i\sqrt{\tau}$. 

Sustained by similar arguments, one can prove as well that 
\begin{align*}
	& E_{2n+1}\left(-\frac{y}{2}\right)  =  \frac{2}{\pi}(-1)^{n+1} \cos\left(\frac{\pi y}{2}\right)  
		\lim_{\lambda\to \pi^{-}} \frac{\partial^{2n+1}}{\partial \lambda^{2n+1}} 
		 \int_{0}^\infty  \e^{-2\sqrt{x} \cos(\lambda/2)} K_{y+1}(2\sqrt{x}) \frac{dx}{\sqrt{x}} ,
		 \ , \ n\in\mathbb{N}_{0},
\end{align*}
which, after \eqref{repre Pntilde in pi over 2}, may be restyled into 
\begin{align} \label{eq Euler y odd}
	& E_{2n+1}\left(-\frac{y}{2}\right)  =  \frac{2}{\pi}(-1)^{n+1} \cos\left(\frac{\pi y}{2}\right) 
		 \int_{0}^\infty   \widetilde{P}_{n}(x) K_{y+1}(2\sqrt{x}) dx ,
		 \ , \ n\in\mathbb{N}_{0},
\end{align}
which leads to \eqref{Eulerodd KLs Pntilde} upon the replacement  $y=-1+2i\sqrt{\tau}$. 

Actually, \eqref{eq Euler y even} and \eqref{eq Euler y odd} are both valid in the vertical strip $-4<\Re y <2$ because of the uniform and absolute convergence of the integral. 

In the meantime, since the Euler polynomials admit the following expansion \cite[(24.2.10)]{NIST}
\begin{equation}\label{Euler itau plus half}
\left\{\begin{array}{ccll}
	\ds E_{2n+2}\left(i\sqrt{\tau}+\tfrac{1}{2}\right)
	&=& \ds  \sum _{k=0}^{n+1} (-1)^k \binom{ 2 n+2 }{  2 k} \frac{E_{2 n+2-2 k}}{2^{2n-2k+2}}\tau ^k
		\\
	\ds \frac{1}{i\sqrt{\tau}}E_{2n+1}\left(i\sqrt{\tau}+\tfrac{1}{2}\right)
	&=& \ds  \sum _{k=0}^{n} (-1)^k \binom{ 2 n+1 }{  2 k+1} \frac{E_{2 n-2 k}}{2^{2n-2k}}\tau ^k
	&\ , \ n\in\mathbb{N}_{0}, 
\end{array}\right. 
\end{equation}
it readily follows from \eqref{Eulereven KLc xPn}
$$
	x P_{n}(x) = \sum _{k=0}^{n+1} (-1)^{n+k+1} \binom{ 2 n+2 }{  2 k} \frac{E_{2 n+2-2 k}}{2^{2n-2k+2}}
		\widetilde{P}_{k}(x)
	\ , \ n\in\mathbb{N}_{0}, 
$$
and, concomitantly, \eqref{Eulerodd KLs Pntilde} implies  
$$
	\widetilde{P}_{n}(x)
	= \sum _{k=0}^{n} (-1)^{n+k} \binom{ 2 n+1 }{  2 k+1} \frac{E_{2 n-2 k}}{2^{2n-2k}}P_{k}(x) 
	\ , \ n\in\mathbb{N}_{0}. 
$$

\subsection{The dual sequences of the MPSs whose $KL_{\{s,c\}}$-transforms are the canonical sequences}\label{sec: Prudnikov}

Despite none of the polynomial sequences $\{ P_{n}(x)\}_{n\geqslant 0}$ and $\{ \widetilde{P}_{n}(x)\}_{n\geqslant 0}$ can be (regularly) orthogonal  (with respect to an $L_{2}$-inner product), we will  show that the first element of the corresponding dual sequence is a regular form, which amounts to same as the guarantee of the existence of an orthogonal polynomial sequence with respect to it. For a more clear understanding, we recall a few concepts, of the utmost importance for this goal.

Any element $u$ of $\mathcal{P}'$ can be written in a series of any dual sequence $\{ \mathbf{v}_{n}\}_{n\geqslant 0}$ of a MPS  $\{P_{n}\}_{n\geqslant 0}$ \cite{MaroniTheorieAlg,MaroniVariations}: 
\begin{equation} \label{u in terms of un}
	u = \sum_{n\geqslant 0} \langle u , P_{n} \rangle \;{u}_{n} \; .
\end{equation}
Differential equations or other kind of linear relations realized by the elements of the dual sequence can be deduced by transposition of those relations fulfilled by the elements of the corresponding MPS, insofar as a linear operator $T:\mathcal{P}\rightarrow\mathcal{P}$  
has a transpose $^{t}T:\mathcal{P}'\rightarrow\mathcal{P}'$ defined by 
\begin{equation}\label{Ttranspose}
	\langle{}^{t}T(u),f\rangle=\langle u,T(f)\rangle\,,\quad u\in\mathcal{P}',\: f\in\mathcal{P}.
\end{equation}
For example, for any form $u$ and any polynomial $g$, let $ Du=u'$ and $gu$ be the forms defined as usual by
$
	\langle u',f\rangle :=-\langle u , f' \rangle \ ,\  \langle gu,f\rangle :=\langle u, gf\rangle ,
$  
where $D$ is the differential operator \cite{MaroniTheorieAlg,MaroniVariations}. Thus, $D$ on forms is minus the transpose of the differential operator $D$ on polynomials.

The properties of the MPSs $\{P_{n}\}_{n\geqslant 0}$ and $\{\widetilde{P}_{n}\}_{n\geqslant 0}$ trigger those of the corresponding dual sequence. 

\begin{lemma} The dual sequences $\{u_{n}\}_{n\geqslant 0}$ and $\{\widetilde{u}_{n}\}_{n\geqslant 0}$ of the two MPSs $\{P_{n}\}_{n\geqslant 0}$ and $\{\widetilde{P}_{n}\}_{n\geqslant 0}$, respectively fulfill 
\begin{align}
	&	(x^{2}u_{0})'' - 3(x u_{0})' + (1-x)u_{0}=0  \label{Dif Eq u0},\\ 
	& 	(x^{2}u_{n+1})'' - 3(x u_{n+1})' + (1-x)u_{n+1}= - u_{n}  \ , \ n\geqslant 0\  , \label{Dif Eq un}
\end{align}
and 
\begin{align}
	&	(x^{2}\widetilde{u}_{0})'' - 2(x \widetilde{u}_{0})' + \left(\tfrac{1}{4}-x\right)\widetilde{u}_{0}=0  
		\label{Dif Eq u0tilde},\\ 
	& 	(x^{2}\widetilde{u}_{n+1})'' - 2(x \widetilde{u}_{n+1})' + \left(\tfrac{1}{4}-x\right)\widetilde{u}_{n+1}
		= - \widetilde{u}_{n}  \ , \ n\geqslant 0\ . \label{Dif Eq untilde}
\end{align}

Moreover,   $(u_{0})_{n}=(n!)^{2}$ and $(\widetilde{u}_{0})_{n}=\prod\limits_{\sigma=0}^{n-1} (\frac{1}{2}+\sigma)^2=2^{-n}((2n-1)!!)^2$,  $   n\in \mathbb{N}_{0}$. 
\end{lemma}

Upon rearrangement, the relations \eqref{Dif Eq u0}-\eqref{Dif Eq un} can be written 
\begin{align*}
	&	x D x D u_{0} - x \ u_{0} =0, \\
	& 	x D x D \, u_{n+1} - x \ u_{n+1}= - u_{n}  \ , \ n\geqslant 0\ . 
\end{align*}

\begin{proof} The action of $u_{0}$ over \eqref{PolySeq Pn} is given by 
$$
	\langle u_{0} , x^{2}P_{n}''(x) +3x P_{n}'(x)+(1-x)P_{n}(x) \rangle =0 \ , \ n\in \mathbb{N}_{0}, 
$$
which, by transposition, on account of \eqref{Ttranspose}, is equivalent to 
$$
	\langle \big( x^{2}u_{0}\big)'' - 3\big( x u_{0}\big)' + (1-x)u_{0} ,  P_{n}\rangle =0 \ , \ n\in \mathbb{N}_{0}, 
$$
providing \eqref{Dif Eq u0}. Likewise, the action of $u_{k+1}$ over \eqref{PolySeq Pn} yields 
$$
	\langle u_{k+1} , x^{2}P_{n}''(x) +3x P_{n}'(x)+(1-x)P_{n}(x) \rangle 
	= - \delta_{n,k} \ , \ n,k\in \mathbb{N}_{0} ,
$$
and, due to \eqref{Ttranspose}, we may write this latter as 
$$
	\langle \big( x^{2}u_{k+1}\big)'' - 3\big( x u_{k+1}\big)' + (1-x)u_{k+1} ,  P_{n}\rangle 
	=  - \delta_{n,k} \ , \ n,k\in \mathbb{N}_{0}. 
$$
By virtue of \eqref{u in terms of un}, the relation \eqref{Dif Eq un} is then a consequence of this latter equality. 

The action of both sides of \eqref{Dif Eq u0} over the sequence $\{x^{n}\}_{n\geqslant0}$ permits to obtain the relation for the moments of $u_{0}$, since we have 
$$
	(u_{0})_{n+1} = (n+1)^{2} (u_{0})_{n}  \ , \ n\in \mathbb{N}_{0}
$$
and thereby $(u_{0})_{n}=(n!)^{2}(u_{0})_{0}=(n!)^{2}$ for $ \ , \ n\in \mathbb{N}_{0}$. 

Replicating this procedure for the MPS $\{\widetilde{P}_{n}\}_{n\geqslant 0}$ based on the starting point \eqref{PolySeq Pntilde}, we conclude that the dual sequence $\{\widetilde{u}_{n}\}_{n\geqslant 0}$ necessarily fulfill the equations \eqref{Dif Eq u0tilde}-\eqref{Dif Eq untilde}. The moments of $\widetilde{u}_{0}$ can be directly computed from \eqref{Dif Eq u0tilde}.  
\end{proof}

\begin{remark} \begin{enumerate}
\item Any MPS $\{B_{n}\}_{n\geqslant0}$ such that $B_{0}=1$ and $B_{n+1}(0)=0$ has the Dirac delta $\delta$ as canonical form. For this reason, the sequence $\{B_{n}(x):=\sum\limits_{k=0}^n T_{E}(n,k)x^k\}_{n\geqslant0}$ has $\delta$ as its canonical form. 

\item The moments of the dual sequence of any MPS $\{B_{n}\}_{n\geqslant0}$ are the coefficients in the expansion of $\{x^{n}\}_{n\geqslant0}$ in terms of $\{B_{n}\}_{n\geqslant0}$. Precisely if $\{b_{n}\}_{n\geqslant0}$ represents the dual sequence of $\{B_{n}\}_{n\geqslant0}$, we then have 
$$
	x^n = \sum_{\nu=0}^n (b_{\nu})_{n} \ B_{\nu}(x) \ , \ n\in\mathbb{N}_{0}.
$$

Consequently, it readily follows from \eqref{xn to Pn} that the moments of the dual sequence 
$\{(u_{n})_{k}\}_{0\leqslant n\leqslant k}$ of the MPS $\{(P_{n})_{k}\}_{0\leqslant n\leqslant k}$ are given by 
$$
	(u_{n})_{k} = (-1)^{n+k} t_{E}(k+1,n+1) ,
$$
with $(u_{0})_{k} = (k!)^2$ for any $k\in\mathbb{N}_{0 }$, and the moments of the dual sequence 
$\{(\widetilde{u}_{n})_{k}\}_{0\leqslant n\leqslant k}$ of the MPS $\{(\widetilde{P}_{n})_{k}\}_{0\leqslant n\leqslant k}$are, according to \eqref{xn to Pntilde}, equal to 
$$
	(\widetilde{u}_{n})_{k} = (-1)^{n+k} t_{O}(k,n) ,
$$
with $(\widetilde{u}_{0})_{k+1}=\prod\limits_{\sigma=0}^{k} (\frac{1}{2}+\sigma)^2=2^{-k-1}((2k+1)!!)^2$ and $(\widetilde{u}_{0})_{0}=1$ for any $k\in\mathbb{N}_{0 }$.
\end{enumerate}
\end{remark}

Despite the non-(regular)orthogonality of  $\{P_{n}\}_{n\geqslant 0}$ with respect to the form $u_{0}$, we cannot exclude the existence of an orthogonal polynomial sequence, say $\{V_{n}\}_{n\geqslant 0}$. Thus, we pose the problem of investigating whether $u_{0}$ or $\widetilde{u}_{0}$ are regular or not. We bring an affirmative answer:

\begin{proposition}	Both of the canonical forms $u_{0}$ and $\widetilde{u}_{0}$ are positive definite, admitting the following integral representations 
\begin{equation*}
	\langle u_{0}, f\rangle = 2\int_{0}^\infty K_{0}(2\sqrt{x}) f(x) dx \ , \ \forall f\in\mathcal{P} 
\end{equation*}
and 
\begin{equation*}
	\langle \widetilde{u}_{0}, f\rangle = 2\int_{0}^\infty K_{0}(2\sqrt{x}) f(x) \frac{dx}{\sqrt{x}} 
	\ , \ \forall f\in\mathcal{P}. 
\end{equation*}

\end{proposition}

\begin{proof} We seek functions $U(x)$ and $\widetilde{U}(x)$ such that 
$$
	\langle u_{0}, f\rangle = \int_{C} U(x) f(x) dx 
	\quad \text{ and } \quad 
	\langle \widetilde{u}_{0}, f\rangle = \int_{\widetilde{C}} \widetilde{U}(x) f(x) dx 
$$ 
respectively hold in a certain domain $C$ and $\widetilde{C}$. 
Since $\langle u_{0} , 1\rangle=\langle \widetilde{u}_{0} , 1\rangle=1\neq 0$, we must have 
\begin{equation}\label{In U neq 0}
	\int_{C}   U(x) dx =\int_{\widetilde{C}} \widetilde{U}(x)   dx=1\neq 0 
\end{equation}
Consider the three polynomials presented in  \eqref{Dif Eq u0} 
$$
	\phi(x)=x^2 \ , \ \psi(x)=-3x \ , \ \chi(x)=1-x \ .
$$
By virtue of \eqref{Dif Eq u0}, we have, for any $f\in\mathcal{P}$
$$\begin{array}{lll}
	0 	&=& \langle \Big( (\phi(x) u_{0})' + \psi(x) u_{0} \Big)' + \chi(x) u_{0} , f (x)\rangle 
		= \langle u_{0}, \phi(x) f''(x) + \psi(x) f'(x) + \chi(x) f(x)  \rangle \vspace{0.2cm}\\
		&=& \ds \int_{C}\left( (\phi(x) U(x))''+(\psi(x) U(x))'+\chi(x)U(x)\right)f(x) dx \\
		&&	- \left.\Big(\phi(x)U(x)f'(x) - (\phi(x)U(x))'f(x) - \psi(x) U(x)f(x)\Big)\right|_{C} 
\end{array}
$$
therefore, $U(x)$ is a function  simultaneously fulfilling 
\begin{align}
	 & \int_{C}\left( (\phi(x) U(x))''+(\psi(x) U(x))'+\chi(x)U(x)\right)f(x) dx =0\quad ,\quad \forall f\in\mathcal{P}
	 	\label{Cond1 for U}\\
	 & \left.\Big(\phi(x)U(x)f'(x) - (\phi(x)U(x))'f(x) - \psi(x) U(x)f(x)\Big)\right|_{C}=0
	 	\quad ,\quad \forall f\in\mathcal{P} \ .
		\label{Cond2 for U}
\end{align}
The first equation implies 
$$
	 (\phi(x) U(x))''+(\psi(x) U(x))'+\chi(x)U(x) = \lambda g(x)
$$
where $\lambda$ is a complex number and $g(x)\neq 0$ is a function representing the null form, that is, a function such that  
$$
	 \int_{C} g(x) f(x) dx =0 \quad , \quad \forall f\in\mathcal{P}.
$$
We begin by choosing $\lambda=0 $ and we search a regular solution of the differential equation 
$$
	 (x^2 U(x))''-(3x U(x))'+ (1-x) U(x) = 0 \ , 
$$
whose general solution is: \ 
$
	y(x) = c_{1} I_{0}(2\sqrt{x}) + c_{2} K_{0}(2\sqrt{x}) \ , \ x\geqslant 0, 
$
for some arbitrary constants $c_{1},c_{2}$ and $y(x)=0$ when $x<0$ \cite{Bateman}. 
As a consequence 
$
	U(x) = \Big\{c_{1} I_{0}(2\sqrt{x}) + c_{2} K_{0}(2\sqrt{x})\Big\}\ , \ x\geqslant 0. 
$
Insofar as $U(x)$ must be a rapidly decreasing sequence (that is, such that $\lim\limits_{x\to +\infty} f(x)U(x)=0$ for any polynomial $f$) simultaneously realizing the condition \eqref{In U neq 0}, we readily conclude that $U(x)=2K_{0}(2\sqrt{x})$, considering its asymptotic behavior \eqref{Knu at infty}-\eqref{K0 at 0} together with the expression for its moments \eqref{MomentsK0}. 
Moreover, for every polynomial $p$ that is not identically zero and is non-negative for all real $x$ we have 
$\ 
\ds	\langle u_{0} ,  p \rangle =  2 \int_{0}^{\infty} \! p(x)  K_{0}(2\sqrt{x}) \ dx > 0 \ 
$ 
and therefore $u_0$ is a {positive-definite} form, which implies the existence of a corresponding MOPS ({\it i.e.}, $u_0$ is a regular form). 

Reiterating the arguments on the latter procedure to find out the function $U(x)$, under the nuance of the distinct  differential equation \eqref{Dif Eq u0tilde} realized by $\widetilde{u}_{0}$, the function $\widetilde{U}(x)$ must be a solution of the differential equation 
$$
	 (x^2 \widetilde{U}(x))''-(2x \widetilde{U}(x))'+ (\tfrac{1}{4}-x) \widetilde{U}(x) = 0 
$$
under the constraints $\lim\limits_{x\to +\infty} f(x) \widetilde{U}(x)=0$ for any polynomial $f$. As a result, $\widetilde{U}(x) = \frac{2}{\sqrt{x}}K_{0}(2\sqrt{x})$ and the fact that 
$\ 
\ds	\langle u_{0} ,  g \rangle =  2 \int_{0}^{\infty} \! g(x)  K_{0}(2\sqrt{x})\frac{dx}{\sqrt{x}} > 0 \ 
$  
for any polynomial $g$ non-negative for all real $x$ and not identically zero, ensures $\widetilde{u}_{0}$ to be a positive definite form (ergo regular). 
\end{proof}

Naturally, the affirmative answer to the regularity of $u_{0}$ and $\widetilde{u}_{0}$ raises the problem of the identification of two polynomial sequences each of them respectively orthogonal to $u_{0}$ and $\widetilde{u}_{0}$. 

Actually, the problem of characterizing the MOPS with respect to $u_{0}$ was already posed by Prudnikov \cite{PrudnikovProblem} and is still an open problem, which can be traced back to his seminal work of 1966  \cite{DitkinPrudnikov}. Therein we may read  the origins of the problem:  the operational calculus associated to the differential operator $\frac{d}{dt}$ gives rise to the Laplace transform having the exponential function as a kernel, which we are going to represent in terms of the Mellin-Barnes integral \cite[Vol.I]{Bateman}
$$
	\e^{-x} = \frac{1}{2\pi i} \int_{a-i\infty}^{a+i\infty} \Gamma(s) x^{-s} ds 
	\ , \quad x,a>0, 
$$
while the operator  $\frac{d}{dt}t\frac{d}{dt}$ leads to the Meijer transform \cite{YakuLuchko} involving the {modified Bessel function} (also known as {MacDonald's function}) $2K_{0}(2\sqrt{x})$ as a positive kernel given by the formula \cite[Vol.II]{Bateman}
$$
	2K_{0}(2\sqrt{x}) = \frac{1}{2\pi i} \int_{a-i\infty}^{a+i\infty} \Gamma^{2}(s) x^{-s} ds
		\ , \quad x,a>0. 
$$
Bearing in mind that $\e^{-x}$ is the weight function of the very classical Laguerre polynomials (of parameter $0$) \cite{ChiharaBook, Temmebook1996},  Prudnikov posed the problem of finding a new sequence of orthogonal polynomials related to the weight $2K_{0}(2\sqrt{x})$. The problem at issue encompasses the characterization of a MOPS, say $\{V_{n}\}_{n\geqslant 0}$, such that 
\begin{equation*}
	\int_{0}^{\infty} 2 K_{0}(2\sqrt{x}) V_{m} (x) V_{n}(x) dx = N_{n} \delta_{n,m} \ , \quad n,m\in\mathbb{N}_{0},
\end{equation*}
where $\delta_{n,m}$ represents the {\it Kronecker symbol}, and with $N_{n}>0$, through the determination of a differential equation fulfilled by $\{V_{n}\}_{n\geqslant0}$ or a relation for the recursive coefficients. The problem was later on broadened to other ultra-exponential weights (see \cite{PrudnikovProblem}) and, regarding the intrinsic difficulties, in the sequel, other approaches were considered, largely focused on the investigation of polynomial sequences being either multiple orthogonal  \cite{AsscheYakubov2000} or $d$-orthogonal \cite{CheikhDouak}  with respect to these ultra-exponential weights. 

However we do not unravel the characterization of this Prudnikov MOPS, in the present work we have enlightened the investigation of this MPS $\{P_{n}\}_{n\geqslant 0}$ whose $KL_{s}$-transform is the canonical sequence and whose canonical form is shared by the Prudnikov MOPS $\{V_{n}\}_{n\geqslant 0}$. An analogous relation occurs between the Bernoulli and the Legendre polynomials - see \cite{MaroniMejri}.

\section{Identities involving central factorials, Euler polynomials and Genocchi numbers} \label{Sec: Number identities}

 This section is mainly devoted to identities between Euler and central factorial numbers and polynomials, that are essentially consequence of the developments made in \S\ref{subsec: rec rel Pn Pntilde}, but worth to drive the attention because, as far as we are concerned, are new in the theory. Some particular cases of the forthcoming identities can be found in \cite{Cigler}. 

The comparison of \eqref{KLxPn Factorials} with \eqref{KLxPn Euler even}  and also of  \eqref{KL c xPn tilde Factorials} with \eqref{KLxPn Euler odd}  brings an identity between the central factorial coefficients of even order and the Euler polynomials which, as far as we are concerned, is new in the theory: 
\begin{equation}\label{CFact 2 to Euler}
	 \sum_{\nu=0}^{n}(-1)^{\nu}T_{E}(n+1,\nu+1) 
             \prod_{\sigma=1}^{\nu+1} (\tau + \sigma^{2})
             =
	  \frac{1}{ i\sqrt{\tau}} 
		\Bigg\{\tau  E_{2n+2}(i\sqrt{\tau}) 
		 + E_{2n+4}(i\sqrt{\tau})\Bigg\}
\end{equation}
and 
\begin{equation}\label{CFact 2 to Euler odd}
	\sum_{\nu=0}^{n}(-1)^{n+\nu}T_{O}(n,\nu)
               	\prod_{\sigma=0}^{\nu}\left(\left(\tfrac{1}{2}+\sigma\right)^2 + \tau\right)
	=  (-1)^{n}\Bigg\{\tau 
		E_{2n+1}(i\sqrt{\tau} )  
		+ E_{2n+3}(i\sqrt{\tau} ) \Bigg\}\ .
\end{equation}

In particular, when  $\tau\to 0+$ in \eqref{CFact 2 to Euler},  we recover the relation already pointed in \cite[pp.74-75]{StanleyVol2} and \cite{Cigler}: 
\begin{equation}\label{Genocchi Central Fact}
        \sum_{\nu=0}^{n}(-1)^{\nu} T_{E}(n+1,\nu+1)
              \big( (\nu+1)! \big)^{2} 
       =  \mathfrak{G}_{2n+4}
       \ , \ n\in\mathbb{N}_{0}. 
\end{equation}

On the other hand, yet the same framework, in the light of \eqref{KL c xn}, from \eqref{Eulereven KLc xPn} we also deduce 
\begin{equation} \label{TE with odd fact}
	 \sum_{\nu=0}^{n}(-1)^{\nu}T_{E}(n+1,\nu+1) 
             \prod_{\sigma=0}^{\nu} (\tau + (\sigma+\tfrac{1}{2})^{2}) 
             = -E_{2n+2}(\tfrac{1}{2}+i\sqrt{\tau}) \ , \ n\in\mathbb{N}_{0},
\end{equation}
and, likewise, due to \eqref{Eulerodd KLs Pntilde} we obtain  
\begin{equation} \label{TO with even fact}
	 \sum_{\nu=0}^{n}(-1)^{\nu}T_{O}(n,\nu) 
             \prod_{\sigma=1}^{\nu} (\tau + \sigma^{2}) 
             = \frac{1}{i\sqrt{\tau}}E_{2n+1}(\tfrac{1}{2}+i\sqrt{\tau})
             \ , \ n\in\mathbb{N}_{0}.
\end{equation}

As a result of , it follows:

\begin{corollary} The set of numbers $\{T_{E}(n,\nu)  \}_{0\leqslant \nu\leqslant n}$ and $\{T_{O}(n,\nu)  \}_{0\leqslant \nu\leqslant n}$ fulfill the following relations: 
\begin{align}
	 \label{Recurrence CentralFact}
	T_{E}(n+1,\nu)  
	&=  \sum_{k=\nu}^{n+1}
          \binom{2n+2}{2k} 
	\frac{ (n+k+2) \mathfrak{G}_{2n-2k+4}  }
			{(2 k+1) (n-k+2)}
			T_{E}(k+1,\nu+1)  
			\ , \ 0\leqslant \nu\leqslant n+1, \\
	\label{Recurrence CentralFact TOdd}
	T_{O}(n,\nu-1)
	&=\sum_{k=\nu}^{n+1} \binom{2 n+2}{2 k} \frac{(n+k+1) \mathfrak{G}_{2n-2k+4} }
		 {2 (n+1) (n-k+2)}  T_{O}(k,\nu)
		 \ , \ 0\leqslant \nu\leqslant n+1,
	\\
	 \label{TE for TO}
	 T_{E}(n+1,\nu)
	& =- \sum _{k=\nu}^{n+1}  \binom{ 2 n+2 }{  2 k} \frac{E_{2 n+2-2 k}}{2^{2n-2k+2}}T_{O}(k,\nu)
		\ , \ 0\leqslant \nu\leqslant n+1,\\
	 \label{TO for TE}
	T_{O}(n,\nu)
	&=  \sum _{k=\nu}^{n} \binom{ 2 n+1 }{  2 k+1} \frac{E_{2 n-2 k}}{2^{2n-2k}} T_{E}(k+1,\nu+1) 
	\ , \ 0\leqslant \nu\leqslant n.
\end{align}
\end{corollary}
\begin{proof}
 From \eqref{Euleres explic even} together with \eqref{CFact 2 to Euler} and on account of \eqref{x k Stirl2 x k,A} it is possible to successively write  
\begin{eqnarray*}
	&& \sum_{\nu=0}^{n}(-1)^{n+\nu}T_{E}(n+1,\nu+1)  
              \prod_{\sigma=1}^{\nu+1} \left(\sigma^2 + \tau\right)
          	  	=    
		 \sum_{k=0}^{n+1}\binom{2n+2}{2k} 
		\frac{ (n+k+2) \mathfrak{G}_{2n-2k+4}  }
			{(2 k+1) (n-k+2)} (-1)^{n+k+1} 
		\tau^{k}\\
           && \qquad = - \sum_{\nu=0}^{n+1}  \sum_{k=\nu}^{n+1}
	          \binom{2n+2}{2k} 
		\frac{ (n+k+2) \mathfrak{G}_{2n-2k+4}  }
				{(2 k+1) (n-k+2)}
				(-1)^{n+\nu}T_{E}(k+1,\nu+1)  
	         			\prod_{\sigma=1}^{\nu} \left(\sigma^2 + \tau\right),
\end{eqnarray*}
whence 
\begin{align*}
	& \sum_{\nu=1}^{n+1}(-1)^{n+\nu-1}T_{E}(n+1,\nu)  
              \prod_{\sigma=1}^{\nu} \left(\sigma^2 + \tau\right) \\
          &
               = - \sum_{\nu=0}^{n+1}  \sum_{k=\nu}^{n+1}
          \binom{2n+2}{2k} 
	\frac{ (n+k+2) \mathfrak{G}_{2n-2k+4}  }
			{(2 k+1) (n-k+2)}
			(-1)^{n+\nu}T_{E}(k+1,\nu+1)  
              \prod_{\sigma=1}^{\nu} \left(\sigma^2 + \tau\right).
\end{align*}
Insofar as $\left\{\prod\limits_{\sigma=1}^{n} \left(\sigma^2 + \tau\right)\right\}_{n\geqslant 0}$forms a basis of  $\mathcal{P}$, by merely equating the coefficients on the left and right hand sides of the latter identity, comes out \eqref{Recurrence CentralFact}. 

The relation \eqref{Recurrence CentralFact TOdd} can be obtained by an analogous procedure via \eqref{CFact 2 to Euler odd}.

Considering \eqref{central fact change basis2} and  \eqref{Euler itau plus half}, the relations \eqref{TE with odd fact}-\eqref{TO with even fact} may be respectively rewritten  as follows: 
\begin{eqnarray*}
	&& \sum_{\nu=0}^{n}(-1)^{\nu}T_{O}(n,\nu) 
             \prod_{\sigma=1}^{\nu} (\tau + \sigma^{2}) 
            =  \sum _{k=0}^{n} (-1)^k \binom{ 2 n+1 }{  2 k+1} \frac{E_{2 n-2 k}}{2^{2n-2k}}
             	 \sum_{\nu=0}^{k}(-1)^{k+\nu}T_{E}(k+1,\nu+1) 
             \prod_{\sigma=1}^{\nu} (\tau + \sigma^{2}) \ , \\
           &&   \sum_{\nu=0}^{n}(-1)^{\nu}T_{E}(n+1,\nu+1) 
             \prod_{\sigma=0}^{\nu} (\tau + (\sigma+\tfrac{1}{2})^{2}) 
          \\ 
          && \qquad\qquad =  \sum _{k=0}^{n+1} (-1)^{k+1} \binom{ 2 n+2 }{  2 k} \frac{E_{2 n+2-2 k}}{2^{2n-2k+2}}
              \sum_{\nu=0}^{k}(-1)^{\nu}T_{O}(k,\nu) 
             \prod_{\sigma=0}^{\nu-1} (\tau + (\sigma+\tfrac{1}{2})^{2}) \ .
\end{eqnarray*}
\end{proof}

We stress that the relations \eqref{TE for TO}-\eqref{TO for TE}, relating the even order central factorial numbers of second kind to those of odd order are alternatives to the one point out in  \cite[p.216]{Riordan2}. 

By virtue of \eqref{CFact 2 to Euler} and on account of \eqref{x k stirl1 x k,A} along with the second equality in \eqref{KLxPn Euler even}, the identity holds 
\begin{equation*} 
	\ds \sum_{\nu=\mu-1}^{n} (-1)^{\nu+\mu+1} T_{E}(n+1,\nu+1) t_{E}(\nu+2,\mu+1)
	=  \binom{2n+2}{2\mu} \frac{ (n+\mu+2) \mathfrak{G}_{2n-2\mu+4}  }
			{(2 \mu+1) (n-\mu+2)} \ , \ 0\leqslant \mu\leqslant n. 
\end{equation*}
Analogously, considering \eqref{x k stirl1 x k,A} together with the second equality in \eqref{KLxPn Euler odd},  the relation \eqref{CFact 2 to Euler odd}  provides 
\begin{equation*}
	\ds \sum_{\nu=\mu-1}^{n} (-1)^{\nu+\mu+1} T_{O}(n,\nu) t_{O}(\nu+1,\mu)
	=  \binom{2n+2}{2\mu} \frac{ (n+\mu+2) \mathfrak{G}_{2n-2\mu+4}  }
			{2( n+1) (n-\mu+2)} \ , \ 0\leqslant \mu\leqslant n. 
\end{equation*}

The due operation of the $KL_{s}$ and $KL_{c}$-transforms over the reciprocal relations \eqref{xn to Pn}-\eqref{xn to Pntilde}, permits to express the central factorials in terms of the Euler polynomials. Precisely, we have:  

\begin{corollary}  For all the integers $n\geqslant 0$, the following identities hold: 
\begin{align}
	\label{factorials to Eulers via s0}\prod_{\sigma=1}^{n+1}\left(\sigma^2 + \tau\right)
		&=\frac{(-1)^{n}}{i\sqrt{\tau}}  \sum_{\nu=0}^{n} t_{E}(n+1,\nu+1) 
		\left\{  \tau E_{2\nu+2} (i\sqrt{\tau}) +E_{2\nu+4} (i\sqrt{\tau})  \right\}  \  , \ n\in\mathbb{N}_{0}\ ,
		\\
	\label{factorials odd to Eulers odd}
	\prod_{\sigma=0}^{n-1}\left((\tfrac{1}{2}+\sigma)^2 + \tau\right)
		&= {(-1)^{n}}   \sum_{\nu=0}^{n} t_{O}(n,\nu) 
		\left\{  \tau E_{2\nu+1} (i\sqrt{\tau}) +E_{2\nu+3} (i\sqrt{\tau})  \right\}  \  , \ n\in\mathbb{N}\ , 
		\\
	\label{factorials odd to Eulers even onehalf}
	\prod_{\sigma=0}^{n}\left((\tfrac{1}{2}+\sigma)^2 + \tau\right)
		&={(-1)^{n}}  \sum_{\nu=0}^{n} t_{E}(n+1,\nu+1) 
		 E_{2\nu+2} (i\sqrt{\tau}+\tfrac{1}{2})   \  , \ n\in\mathbb{N}_{0},
		 \\
	\label{factorials even to Eulers odd onehalf}
	\prod_{\sigma=1}^{n}\left( \sigma^2 + \tau\right)
		&=\frac{(-1)^{n}}{i\sqrt{\tau}}  \sum_{\nu=0}^{n} t_{O}(n,\nu) 
		 E_{2\nu+1} (i\sqrt{\tau}+\tfrac{1}{2}) \  , \ n\in\mathbb{N}. 
	\end{align}
\end{corollary}

\begin{proof} By operating with the $KL_{s}$ and $KL_{c}$-transforms on both sides of \eqref{xn to Pn} multiplied by $x$ and then taking into account \eqref{KL s xn} together with the first equality in \eqref{KLxPn Euler even} and \eqref{Eulereven KLc xPn}, we conclude \eqref{factorials to Eulers via s0} and \eqref{factorials odd to Eulers even onehalf} respectively. 

Mimicking the above technique over the relation  \eqref{xn to Pntilde}, then, on account of  \eqref{KL c xn} together with the first equality in \eqref{KLxPn Euler odd} and \eqref{Eulerodd KLs Pntilde}, we come out with \eqref{factorials odd to Eulers odd} and  \eqref{factorials even to Eulers odd onehalf}. 
\end{proof}

This last result also provides several identities between the central factorials of first kind of even and odd orders. There are many other relations that will not be listed in this text, but can be easily derived from the aforementioned equalities. For instance,  since the left hand side of \eqref{factorials to Eulers via s0} can be expanded in terms of the monomials  through \eqref{x k stirl1 x k,A}, while the right hand side admit the expansion in  \eqref{KLxPn Euler even}, then  
	$$\left\{\begin{array}{l} 
		\ds (-1)^{\mu+1} t_{E}(n+2,\mu+1) 
			= \sum_{\nu=\mu-1}^n (-1)^{\nu} t_{E}(n+1,\nu+1) 
			\binom{2\nu+2}{2\mu} \frac{ (\nu+\mu+2) \mathfrak{G}_{2\nu-2\mu+4}  }
			{(2 \mu+1) (\nu-\mu+2)} \vspace{0.2cm}\\
		\ds \left[(n+1)!\right]^2
		=  \sum_{\nu=0}^n (-1)^{n+\nu} t_{E}(n+1,\nu+1)  \mathfrak{G}_{2\nu+4} \ .
		\end{array}\right. $$

\section{Final considerations}

While looking at the $KL_{s}$-transformed polynomial sequence of classical polynomials of Laguerre or Hermite we come to the conclusion that they are necessarily $d$-orthogonal sequences. The existence of such examples gives grounds to determine all the MOPSs whose image by the action of the $KL_{s}$-transform are $d$-MOPSs. The analysis beneath the answer involve a considerable amount of computations and because of this we defer the study for a forthcoming work. Nevertheless, we announce that if a MOPS is such that the corresponding $KL_{s}$-transform is a $d$-orthogonal sequence, then it is necessarily a semiclassical polynomial sequence and $d$ must be an even number greater than $2$. 
Thus, it is clear that the image of Prudnikov's MOPS cannot be a $d$-MOPS. Indeed, we are still working toward the characterization of these polynomials and we believe that  properties  of the considered sequences $\{P_{n}\}_{n\geqslant 0}$ and $\{\widetilde{P}_{n}\}_{n\geqslant 0}$  can play an important role. 

On the other hand, the {\it Continuous Dual Hahn} and {\it Wilson} orthogonal polynomials (see \cite[\S18.21]{NIST}) are images of two corresponding non orthogonal MPSs whose explicit expressions can be directly obtained. The structural and differential relations are left to be worked on.

\bibliographystyle{amsplain}

\end{document}